\documentclass[10pt]{article}
\usepackage{geometry}                
\newcommand{\hide}[1]{}

\usepackage[title]{appendix}

\usepackage{xcolor}
\usepackage{url}
\usepackage[linkbordercolor={1 0 0}]{hyperref}
\usepackage{mathtools}
\usepackage{amssymb}
\usepackage{algorithm}
\usepackage[noend]{algpseudocode}
\usepackage{authblk}

\hide{
\usepackage{amscd}
\usepackage{amsfonts}
\usepackage{amsmath}
\usepackage{amssymb}
\usepackage{amsthm}
\usepackage{cases}		
\usepackage{cutwin}
\usepackage{enumerate}
\usepackage{epstopdf}
\usepackage{graphicx}
\usepackage{ifthen}
\usepackage{lipsum}
\usepackage{mathrsfs}	
\usepackage{multimedia}
\usepackage{wrapfig}
}
\usepackage{amsthm}
\newtheorem{theorem}{Theorem}[section]

\newtheorem{lemma}[theorem]{Lemma}
\newtheorem{proposition}[theorem]{Proposition}
\newtheorem{assumption}[theorem]{Assumption}
\newtheorem{definition}[theorem]{Definition}
\newtheorem{remark}[theorem]{Remark}
\newlength\tindent
\title{ Cubic Regularized Newton Method for Saddle Point Models: \\ a Global and Local Convergence Analysis }
\author{Kevin Huang\thanks{Department of Industrial and System Engineering, University of Minnesota, huan1741@umn.edu},
Junyu Zhang\thanks{Department of Industrial and System Engineering, University of Minnesota, zhan4393@umn.edu},
Shuzhong Zhang\thanks{Department of Industrial and System Engineering, University of Minnesota, zhangs@umn.edu}}

\begin{document}

\maketitle

\begin{abstract}
    In this paper, we propose a cubic regularized Newton (CRN) method for solving convex-concave saddle point problems (SPP).
    At each iteration,
    a cubic regularized saddle point subproblem is constructed and solved, which provides a search direction for the iterate. With properly chosen stepsizes,
    the method is shown to converge to the saddle point with global linear and local superlinear convergence rates, if the saddle point function is gradient Lipschitz and strongly-convex-strongly-concave.
    In the case that the function is merely convex-concave, we propose a homotopy continuation (or path-following) method. Under a Lipschitz-type error bound condition, we present an iteration complexity bound of $\mathcal{O}\left(\ln \left(1/\epsilon\right)\right)$ to reach an $\epsilon$-solution through a homotopy continuation approach, and the iteration complexity bound becomes $\mathcal{O}\left(\left(1/\epsilon\right)^{\frac{1-\theta}{\theta^2}}\right)$ under a H\"{o}lderian-type error bound condition involving a parameter $\theta$ ($0<\theta<1$).

    \vspace{3mm}
\end{abstract}


\section{Introduction}
In this paper, we aim to solve the following minimax saddle point model:
\begin{equation}
    \label{intro-1}
    \min\limits_{x\in\mathcal{X}}\max\limits_{y\in\mathcal{Y}} f(x,y).
\end{equation}
Such model has applications in various fields including game theory \cite{bacsar1998dynamic,roughgarden2010algorithmic,von2007theory}, robust optimization \cite{ben2009robust}, and distributionally robust optimization \cite{abadeh2015distributionally,gao2016distributionally}, among others. It also arises in the context of machine learning and deep learning in recent years, for instance, generative adversarial network (GAN)~\cite{goodfellow2014generative,martin2017wasserstein,gidel2018variational}.

In this paper, we consider model \eqref{intro-1} under the convex-concave setting, where $f(\cdot,y)$ is convex for any fixed $y$ and $f(x,\cdot)$ is concave for any fixed $x$. For solving the above minimax saddle point model, there have been a number of recent papers in the literature, including the extra-gradient method \cite{korpelevich1976extragradient,tseng1995linear}, the mirror-prox algorithm \cite{nemirovski2004prox}, the dual extrapolation method \cite{nesterov2007dual}, and the accelerated proximal gradient method
\cite{tseng2008accelerated}. These algorithms typically are shown to have an iteration complexity bound of $\mathcal{O}(1/\epsilon)$, which is optimal \cite{ouyang2019lower} for algorithms using only first-order oracles in the convex-concave saddle point setting. If we further restrict the function 
to be strongly convex and strongly concave, then the algorithms analyzed include proximal point method \cite{rockafellar1976monotone,tseng1995linear},
extrapolation 
\cite{nesterov2006solving,gidel2018variational}, and optimistic gradient descent ascent \cite{mokhtari2019unified}. The algorithms proposed in these papers were shown to possess an iteration complexity bound of  $\mathcal{O}(\kappa\ln(1/\epsilon))$, where $\kappa = L/\mu$ denotes the condition number.
In \cite{lin2020near}, the authors further specify the different strong convexity/concavity modulus $\mu_x$ and $\mu_y$ in the problem parameters, and an accelerated proximal point algorithm (APPA) is derived to yield an $\mathcal{O}\left(\sqrt{\kappa_x\kappa_y}\ln^3(1/\epsilon)\right)$ iteration bound, with $\kappa_x = L/\mu_x$ and $\kappa_y = L/\mu_y$. This complexity is optimal up to a logarithmic factor in that it matches the lower bound result established in \cite{junyu2019}.

We shall remark here that the above mentioned algorithms are all first-order methods. In fact, the majority of research on the saddle point models, or more generally, the variational inequality (VI) problems:
\begin{equation*}
    \label{intro-2}
    \textrm{find $z^*$ such that} \hspace{3mm}
    \langle F(z^*),z-z^*\rangle\geq0, \hspace{3mm}\forall z\in\mathcal{Z}
\end{equation*}
focus on using the first-order oracles. In the meanwhile, we note that for optimization models there have been efficient high-order methods such as Newton's method \cite{nocedal2006numerical}, the cubic regularized Newton method \cite{nesterov2006cubic, nesterov2008accelerating} and the higher-order methods \cite{nesterov2018implementable}, and their accelerated variants. This discrepancy of developments in the second- (or higher-) order methods between optimization and saddle point/VI problems, may be attributed to the lack of an effective ``merit function''. A merit function $m(z)$ is a function that measures the progress of the algorithm in the process, and $m(z^k)\to 0$ if and only if $z^k\to z^*$. For (convex) optimization, the natural 
merit function is the objective function $f(x)$ itself (or $f(x)-f(x^*)$ after normalization), since the decrease of the objective value measures the progress of the algorithm.
Unfortunately, there is no such natural merit function in the case of minimax saddle point problems. 
Instead, the duality gap given by $\max\limits_{y\in\mathcal{Y}}f(x^k,y)-\min\limits_{x\in\mathcal{X}}f(x,y^k)$
is often used as a merit measure. Another possible merit function could be the direct measurement of the distance from the optimal solution $\|x^k-x^*\|^2+\|y^k-y^*\|^2$. Moving further towards VI problems, besides using the distance $\|z^k-z^*\|^2$, one could also design other less straightforward merit functions. For example, \cite{nemirovski2004prox, nesterov2007dual} for monotone VI problems or \cite{nesterov2006solving} for strongly monotone VI problems:
\begin{eqnarray}
     &    & (F\textrm{ monotone})\hspace{2mm} m(z)=\max\limits_{u\in\mathcal{Z}}\langle F(u),z-u\rangle, \label{intro-3} \\
     &    & (F \textrm{ strongly monotone})\hspace{2mm} m(z)=\max\limits_{u\in\mathcal{Z}}\langle F(u),z-u\rangle+\frac{\mu}{2}\|u-z\|^2.
\end{eqnarray}

Remark that the second-order methods for optimization (such as Newton's method or the cubic regularized Newton method) can be analyzed through the second-order information of the objective function value. Unfortunately, the second-order information of the above merit functions is unavailable in the saddle point/VI settings.
For this reason, the usual merit functions as mentioned above work well with the first-order methods but not for the second- or higher-order methods. Designing an appropriate merit function becomes critical in developing and analyzing higher-order methods for the saddle point/VI problems.

Technical difficulties aside, a natural question remains: {\it Can one develop second-order method to solve saddle point problem with global convergence guarantee?} The answer is affirmative. The authors of \cite{taji1993globally} proposed to solve strongly monotone VI problems using Newton's method, while establishing the global convergence and local quadratic convergence. In their approach, the corresponding merit function is:
\begin{equation}
    \label{intro-4}
    m(z)=\max\limits_{u\in\mathcal{Z}}\langle F(z),z-u\rangle-\frac{\mu}{2}\|z-u\|^2.
\end{equation}
Note the difference between \eqref{intro-4} and \eqref{intro-3}.
In this paper we propose a cubic regularized Newton (CRN) method to solve strongly-convex-strongly-concave saddle point problems (SPP), and we shall analyze its performance by a merit function. We propose to use the squared norm of the gradient as the merit function to measure the progress of algorithm and we shall give an iteration complexity of $\mathcal{O}\left((\kappa^2+\frac{\kappa L_2}{\mu})\ln(1/\epsilon)\right)$ to obtain an $\epsilon$-saddle point solution, as well as local quadratic convergence. In addition, we propose a Newton method to solve the CRN saddle point subproblem. Finally, we propose to combine a parameterized homotopy continuation (or path-following) approach with the CRN method to solve a class of convex-concave saddle point problems satisfying a certain error bound condition. Homotopy continuation/path-following approach is popular in computing fixed points/optimization. For references on homotopy continuation, see e.g.~\cite{Judd98}, and for the barrier-based path-following methods for convex optimization, see e.g.~\cite{wright1997primal,renegar2001mathematical}.

Recently, the authors of \cite{zhang2020newton} propose two Newton-type algorithms for solving non-convex-non-concave saddle point problems. The first one is gradient-descent-Newton (GDN) method, which updates variable $x$ with gradient descent and $y$ with Newton update. The second one is complete Newton (CN) method, which updates $x$ with envelope Newton update (see definition in \cite{zhang2020newton}) and $y$ with regular Newton update. For the latter, a local superlinear convergence to a strict local minimax saddle point is shown. Note that in this paper we are considering (strongly) convex-concave saddle point problems and established \textit{both} global convergence guarantee and local superlinear/quadratic convergence.  These methods are different from ours, in that we update $x,y$ simultaneously through solving a CRN saddle point subproblem.

The rest of the paper is organized as following. Section~\ref{prelim} establishes preliminaries such as problem descriptions, necessary assumptions, and our merit function; Section~\ref{algorithm} presents the CRN subproblem and analyzes the global convergence property with guaranteed global linear convergence and local quadratic convergence; Section~\ref{subproblem} analyzes the CRN subproblem more closely; Section~\ref{non-strongly} presents a path-following method combined with CRN method for solving a class of convex-concave saddle point problems; Section~\ref{numerical} provides some numerical experiments; Section~\ref{conclusion} concludes the paper. Finally, longer proofs are relegated to the appendix.

\section{Preliminaries} \label{prelim}
Consider the following problem:
\begin{equation}
\label{eq:1-1}
    \min\limits_{x\in\mathbb{R}^n}\max\limits_{y\in\mathbb{R}^m} f(x,y),
\end{equation}
where the function $f$ is twice continuously differentiable. 
\begin{assumption}
\label{ass:2-1}
Function $f$ is strongly convex in $x$ and strongly concave in $y$ with modulus $\mu$.
\end{assumption}
Define $z = (x;y)$ and the following operator:
\begin{equation*}
\label{eq:1-2}
    \begin{array}{cc}
         F(z) =
        \begin{pmatrix}
        \nabla_xf(x,y) \\ -\nabla_yf(x,y)
        \end{pmatrix}.
    \end{array}
\end{equation*}
Due to the strong convexity and strong concavity of $f$, $z^*$ is the unique saddle point of problem~\ref{eq:1-1} if and only if $F(z^*)=0$. Next, we also assume the Lipschitz continuity of the operator $F(z)$ and its Jacobian matrix $\nabla F(z)$. Note that the norm $\|\cdot\|$ denotes $\ell_2$ norm for vectors and largest singular value for matrices.
\begin{assumption}
\label{ass:2-2}
The  largest singular value of $\nabla F(z)$ is upper bounded by $L$ for all $z$, which implies that $F(z)$ is Lipschitz continuous with parameter $L$:
\begin{equation*}
    \label{eq:ad01}
    \|F(z)-F(z')\|\leq L\|z-z'\|, \hspace{3mm}\forall z,z'.
\end{equation*}
\end{assumption}

\begin{assumption}
\label{ass:2-3}
The Jacobian matrix $\nabla F(z)$ is Lipschitz continuous with parameter $L_2$:
\begin{equation*}
\label{eq:1-5}
    \|\nabla F(z)-\nabla F(z')\|\leq L_2\|z-z'\|,\hspace{3mm} \forall z,z'.
\end{equation*}
\end{assumption}


For analyzing our proposed algorithm to 
solve problem \eqref{eq:1-1}, we introduce the following merit function:
\begin{equation*}
\label{eq:1-6}
    \begin{array}{ll}
         m(z) := \frac{1}{2}\|F(z)\|^2=\frac{1}{2}(\|\nabla_xf(x,y)\|^2+\|\nabla_yf(x,y)\|^2), \\
    \end{array}
\end{equation*}
whose gradient is given by:
\begin{equation*}
\label{eq:1-7}
    \nabla m(z)=\nabla F(z)^{\top}F(z)=
    \begin{pmatrix}
    \nabla_{xx}^2f(x,y)\nabla_xf(x,y)+\nabla^2_{xy}f(x,y)\nabla_yf(x,y) \\
    \\
    \nabla^2_{yy}f(x,y)\nabla_yf(x,y)+\nabla^2_{yx}f(x,y)\nabla_xf(x,y)
    \end{pmatrix}.
\end{equation*}
The next lemma stipulates the Lipschitz continuity of $\nabla m(z)$ based upon the previous assumptions.

\begin{lemma}
\label{lem:2-5}
Define the level set $\mathcal{Z}: = \{z: m(z)\leq m(z^0)\}$ and denote $D := \max\{\|z-z^0\|: z\in\mathcal{Z}\}$. Then the gradient $\nabla m(z)$ is Lipschitz continuous within $\mathcal{Z}$. That is, for all $z,z'\in\mathcal{Z}$,
\begin{equation*}
\label{eq:1-8}
    \|\nabla m(z)-\nabla m(z')\|\leq L_m\|z-z'\|\qquad\mbox{with}\qquad L_m=L^2+L_2LD.
\end{equation*}
\end{lemma}
\begin{proof} 
We have
 \begin{eqnarray*}
         & & \|\nabla m(z')-\nabla m(z)\| \\
         & = &\|\nabla F(z')^{\top}F(z')-\nabla F(z)^{\top}F(z)\| \\
         & = & \|\nabla F(z')^{\top}F(z')-\nabla F(z)^{\top}F(z')+\nabla F(z)^{\top}F(z')-\nabla F(z)^{\top}F(z)\| \\
         & \leq & \|\nabla F(z')-\nabla F(z)\|\|F(z')\|+\|\nabla F(z)\|\|F(z')-F(z)\| \\
         & = & \|\nabla F(z')-\nabla F(z)\|\|F(z')-F(z^*)\|+\|\nabla F(z)\|\|F(z')-F(z)\| \\
         & \leq & (L_2\|z'-z\|)\cdot(L\|z'-z^*\|)+L^2\|z'-z\| \\
         & \leq & (L_2LD+L^2)\|z'-z\|=L_m\|z'-z\|,
\end{eqnarray*}
where we used $F(z^*) = 0$.
\end{proof}

In the next proposition we shall establish the connection between $m(z)$ and the duality gap, whose proof is relegated to Appendix~\ref{A1}.
\begin{proposition}
\label{th:2-6}
For problem \eqref{eq:1-1} and any point $z = (x;y)$, the duality gap and the merit function satisfy the following relationship
\begin{equation*}
    \frac{\mu}{L^2}m(z)\leq \max\limits_{y'\in\mathbb{R}^m} f(x,y')-\min\limits_{x'\in\mathbb{R}^n} f(x',y) \leq \frac{L}{\mu^2}m(z).
\end{equation*}
\end{proposition}
Hence we conclude that the two measurements: the merit function $m(z)$ and the duality gap are of the same order of magnitude, and the convergence in one measure implies the convergence of the other.

\section{Algorithm CRN-SPP and Its Convergence Analysis} \label{algorithm}
\subsection{Cubic regularized Newton method for saddle point problem} \label{CRN-sub}
In this section, we present our newly proposed algorithm (CRN-SPP) as sketched in Algorithm \ref{alg:01} below. For ease of notation, denote
$g^k_x = \nabla_x f(x^k,y^k)$, $g_y^k = \nabla_y f(x^k,y^k)$,
 $H^k_{xx} = \nabla_{xx}^2f(x^k,y^k)$, $H^k_{yy} = \nabla_{yy}^2f(x^k,y^k)$, $H^k_{xy} = \nabla_{xy}^2f(x^k,y^k)$, and
$$
g^k = \nabla f(x^k,y^k) = \begin{pmatrix}
g^k_x\\g_y^k
\end{pmatrix}
\qquad\mbox{and}\qquad
H^k = \nabla^2 f(x^k,y^k) = \begin{pmatrix}
H_{xx}^k& H_{xy}^k\\(H_{xy}^k)^{\top} &H^k_{yy}
\end{pmatrix}.
$$
At each iteration \textit{$k$}, we solve the following \textit{saddle point subproblem}:
\begin{equation}
\label{eq:1-11}
\begin{array}{rl}
     \min\limits_{x\in\mathbb{R}^n}\max\limits_{y\in\mathbb{R}^m} & f(z^k) \!+\! \langle g^k,z\!-\!z^k\rangle \!+\!\frac{1}{2}(z\!-\!z^k)^{\top} H^k(z\!-\!z^k)\!
           +\!\frac{\gamma^k}{3}\|x\!-\!x^k\|^3 \!-\!\frac{\gamma^k}{3}\|y\!-\!y^k\|^3\! \\
           & =: f_k(x,y;\gamma^k)
\end{array}
\end{equation}
where $\gamma^k>0$ is a parameter that one chooses at iteration $k$.
\begin{algorithm}
	\caption{CRN-SPP($f,z^0,\epsilon,\bar{\gamma},\rho,\alpha$)}
	\begin{algorithmic}[1]
		\Require Input $\epsilon$, $\bar \gamma>0$; $\rho,\alpha\in(0,1)$, $f$ satisfies Assumption \ref{ass:2-1}, \ref{ass:2-2}, \ref{ass:2-3}.
		\While{$m(z_k)>\epsilon$} {}
		\State $\gamma^k \leftarrow \bar\gamma$
		\While{{\bf true}} {}
		\State Solve the subproblem $(\Tilde{x}^{k+1},\Tilde{y}^{k+1}) = \arg\min_x\max_y f_k(x,y;\gamma^k)$
		\If{$\gamma^k(\|u^k\| + \|v^k\|)>\mu$}
		\State $\gamma^k \leftarrow\rho\cdot\gamma^k$
		\Else \,\,\,\,\textbf{break}.
		\EndIf
		\EndWhile\vspace{0.15cm}
		
		\State Denote $d_k\leftarrow [\Tilde{x}^{k+1}-x^k;\Tilde{y}^{k+1}-y^k]$
		\If{$m(z^k+\alpha d^k)<m(z^k+d^k)$}
		\State $z^{k+1}\leftarrow z^k+\alpha d^k$
		\ElsIf{$m(z^k+\alpha d^k)\geq m(z^k+d^k)$}
		\State $z^{k+1}\leftarrow z^k+d^k$
		\EndIf
		\State $k\leftarrow k+1$
		\EndWhile
		\State \textbf{return} $z^k$
	\end{algorithmic}
	\label{alg:01}
\end{algorithm}

Let $(\Tilde{x}^{k+1},\Tilde{y}^{k+1})$ be the solution to the subproblem \eqref{eq:1-11}, and denote $u^k=\Tilde{x}^{k+1}-x^k$ and $v^k=\Tilde{y}^{k+1}-y^k$, then it satisfies the first-order stationarity condition:
\begin{equation}
    \label{eq:1-14}
    \left\{
    \begin{array}{ll}
         g_x^k+H_{xx}^ku^k+\gamma^k\|u^k\|u^k+H_{xy}^kv^k = 0, \\
         g_y^k\!+\!H^k_{yy}v^k\!-\!\gamma^k\|v^k\|v^k\!+\!(H_{xy}^k)^{\top}\!u^k=0.
    \end{array}
    \right.
\end{equation}

To solve the cubic regularized saddle point subproblem, we propose a Newton method CRN-sub ({Algorithm \ref{alg:02}}) in Section \ref{subproblem} based on solving the stationarity system \eqref{eq:1-14}. In addition, we require that $\gamma^k(\|u^k\|+\|v^k\|) \leq \mu$, which we later will prove to be satisfiable. This requirement is mainly for guaranteeing the descent of the merit function $m(z^k)$. Furthermore, we also make a comparison between taking the $\alpha$ step and the unit step in terms of merit function. This guarantees the global linear convergence and the eventual local quadratic convergence.

\subsection{Global linear convergence}
The following propositions are straightforward.
\begin{proposition}
\label{pro:3-1}
Suppose that $f$ is differentiable. Under Assumption~\ref{ass:2-1},
$z^*=(x^*;y^*)$ is the unique solution to problem \eqref{eq:1-1} if and only if $m(z^*)=0$.
\end{proposition}
Denote $d^k=
(u^k;v^k)$ to be the \textit{update direction} at the \textit{k}-th iteration. The following proposition states that any fixed point solution to subproblem \eqref{eq:1-11} is the unique solution to problem \eqref{eq:1-1}.
\begin{proposition}
\label{pro:3-2}
If $d^k=0$, i.e., $(\Tilde{x}^{k+1},\Tilde{y}^{k+1})=(x^k,y^k)$, then $(x^k,y^k)$ is the unique solution to problem \eqref{eq:1-1}.
\end{proposition}
In the following, we show that the update direction $d^k$ is a descent direction at $z^k$, with respect to 
the merit function $m(z)$ for $\gamma^k$ small enough. The proof is presented in Appendix~\ref{A2}.

\begin{proposition}[Gradient related direction $d^k$]
\label{pro:3-3}
For saddle point subproblem \eqref{eq:1-11}, if we choose $\gamma^k$ small enough then it follows that 
\begin{equation}
\label{eq:cond}
\gamma^k(\|u^k\|+\|v^k\|)<\mu,
\end{equation}
where $\mu$ is the strongly convex/concave modulus of $f(x,y)$. Consequently, $d^k$ is a descending direction w.r.t. the merit function $m(z)$ at point $z^k$; that is,
\begin{equation*}
\label{eq:2-1}
    \langle \nabla m(z^k),d^k\rangle\leq -\frac{\mu^2}{2}\|d^k\|^2.
\end{equation*}
\end{proposition}
Therefore, we can incorporate a stepsize $\alpha>0$ at each iteration together with the direction $d^k$ to form the sequence $\{z^k\}$ according to Algorithm~\ref{alg:01}, which is monotonically descending in terms of the merit function.

\begin{proposition}[Sufficient Descent in $m(z^k)$]
\label{pro:3-4}
With constant stepsizes $\alpha=\frac{\mu^2}{2L_m}<1$,
the sequence $\{m(z^k)\}$ generated by Algorithm \ref{alg:01} satisfies
\begin{equation}
\label{eq:2-13}
m(z^{k+1})-m(z^k)\leq -\frac{\mu^4}{8L_m}\|d^k\|^2.
\end{equation}
\end{proposition}
\begin{proof}
By the so-called descent lemma and Proposition \ref{pro:3-3}, we have
\begin{eqnarray*}
        m(z^k\!\!+\!\alpha d^k)\!-\!m(z^k) \leq  \alpha\nabla m(z^k)^{\top}d^k\!+\!\frac{L_m}{2}\alpha^2\|d^k\|^2  \leq \! -\frac{\alpha\mu^2}{2}\|d^k\|^2\!+\!\frac{L_m}{2}\alpha^2\|d^k\|^2 \!=\! -\frac{\mu^4}{8L_m}\|d^k\|^2.
\end{eqnarray*}
According to the update rule in Algorithm \ref{alg:01},
$$m(z^{k+1})  = \min\{m(z^k+\alpha d^k), m(z^k+ d^k)\} \leq m(z^k+\alpha d^k) \leq m(z^k) -\frac{\mu^4}{8L_m}\|d^k\|^2.$$
Rearranging the terms proves the proposition.
\end{proof}

As a result, we have the following iteration complexity for our algorithm.

\begin{theorem}[Iteration Complexity]
\label{th:3-7}
Let $\{z^k\}$ be generated by Algorithm \ref{alg:01}, with $\alpha = \frac{\mu^2}{2L_m}$ and $L_m = L^2 + LL_2D$, then the sequence $\{m(z^k)\}$ converges linearly to 0:
\begin{equation*}
\label{eq:2-15}
    m(z^{k+1})\leq \Big(1-\frac{\mu^2}{6L_m}\Big)^2m(z^k).
\end{equation*}
As a result, it takes at most $\mathcal{O}\left((\kappa^2+\kappa\cdot\frac{ L_2}{\mu})\ln(\frac{1}{\epsilon})\right)$ iterations to find a point $\bar z$ with $m(\bar z)\leq \epsilon$, where $\kappa=\frac{L}{\mu}$ is the condition number.
\end{theorem}
The proof of this theorem is presented in Appendix~\ref{A3}.  Note that the iteration complexity bound in Theorem~\ref{th:3-7} not only depends on the condition number $\kappa$, but also on $L_2/\mu$. 
Moreover, the dependency on the condition number for this method is in general worse than the first-order methods developed in the literature: $\mathcal{O}\left(\kappa\ln(1/\epsilon)\right)$. However, in the next section we will develop the local quadratic convergence for this method, which is not achievable for any first-order method.


\subsection{Local Convergence Analysis}
Since we solve the CRN subproblem \eqref{eq:1-11} at each iteration, it is natural to analyze its local convergence property. The following theorem states that a local quadratic convergence holds for Algorithm \ref{alg:01} to solve problem \eqref{eq:1-1}.

\begin{theorem}[Local quadratic convergence]
\label{th:3-9}
Let $\{z^k\}$ be generated by Algorithm \ref{alg:01} with $\bar{\gamma}=\frac{L_2\mu^2}{2L^2}$, then there exists a constant $K>0$ such that for all $k\geq K$ we have:
\begin{equation}
    \label{eq:local}
    \|z^{k+1}-z^*\|\leq \frac{LL_2}{\mu^2}\|z^k-z^*\|^2.
\end{equation}
\end{theorem}
\begin{proof}

By Theorem \ref{th:3-7}, there exists a constant $K>0$ such that for all $k\geq K$ we have:
\begin{equation}
    \label{eq:local-2}
    \|z^k-z^*\|\leq \frac{\mu^2}{LL_2}.
\end{equation}

Let us first consider the case of unit step and denote $\tilde{z}^{k+1}=z^k+d^k$.
Also note that:
\begin{equation*}
    \nabla F(z^k)=\begin{pmatrix}
H_{xx}^k& H_{xy}^k\\-(H_{xy}^k)^{\top} &-H^k_{yy}
\end{pmatrix}.
\end{equation*}
Therefore we can rewrite (\ref{eq:1-14}) into:
\begin{equation}
    \label{eq:stat}
    F(z^k)+\nabla F(z^k)d^k+\gamma^k\begin{pmatrix}\|u^k\|u^k\\\|v^k\|v^k\end{pmatrix}=0.
\end{equation}
Rearranging the terms in the above equation and using $F(z^*)=0$ we have
\begin{equation*}
    \label{eq:4-3}
    \begin{array}{ll}
         \nabla F(z^k)(\tilde{z}^{k+1}-z^*) = F(z^*)-F(z^k)-\nabla F(z^k)(z^*-z^k)- \gamma^k\begin{pmatrix}\|u^k\|u^k\\\|v^k\|v^k\end{pmatrix}.
    \end{array}
\end{equation*}
Note that
\begin{equation*}
    \mu\|\tilde{z}^{k+1}-z^*\|^2\leq (\tilde{z}^{k+1}-z^*)^{\top}\nabla F(z^k)(\tilde{z}^{k+1}-z^*)\leq \|\tilde{z}^{k+1}-z^*\|\cdot\|\nabla F(z^{k})(\tilde{z}^{k+1}-z^*)\|.
\end{equation*}
Therefore
\begin{equation}
    \label{eq:4-4}
    \begin{array}{ll}
         \|\tilde{z}^{k+1}-z^*\| & \leq \frac{1}{\mu}\|\nabla F(z^k)(\tilde{z}^{k+1}-z^*)\| \\
         & \leq \frac{1}{\mu}\|F(z^*)-F(z^k)-\nabla F(z^k)(z^*-z^k)\| + \frac{\gamma^k}{\mu}(\|u^k\|^2+\|v^k\|^2) \\
         & \leq \frac{L_2}{2\mu}\|z^k-z^*\|^2+\frac{\Bar{\gamma}}{\mu}\|d^k\|^2. \\
    \end{array}
\end{equation}
Now we need to bound the term $\|d^k\|^2$. From (\ref{eq:stat}) we have:
\begin{equation*}
    \|F(z^k)^{\top}d^k\|\geq -F(z^k)^{\top}d^k=(d^k)^{\top}\nabla F(z^k)d^k+\gamma^k(\|u^k\|^3+\|v^k\|^3)\geq \mu\|d^k\|^2,
\end{equation*}
which gives us
\begin{equation*}
    \|d^k\|\leq \frac{1}{\mu}\|F(z^k)\|\leq \frac{L}{\mu}\|z^k-z^*\|.
\end{equation*}
Using this result in (\ref{eq:4-4}) we get
\begin{equation}
\label{eq:unit-quad}
    \|\tilde{z}^{k+1}-z^*\|\leq \left(\frac{L_2}{2\mu}+\frac{\bar{\gamma}L^2}{\mu^3}\right)\|z^k-z^*\|^2=\frac{L_2}{\mu}\|z^k-z^*\|^2.
\end{equation}
Note that if unit step is accepted in Algorithm \ref{alg:01}, that is, $z^{k+1}=\tilde{z}^{k+1}$, then we have (\ref{eq:unit-quad}), which implies (\ref{eq:local}) under condition (\ref{eq:local-2}).

On the other hand, if $\alpha$ step is accepted, denoting $\hat{z}^{k+1}=z^k+\alpha d^k$, we have
\begin{equation*}
    \|\hat{z}^{k+1}-z^*\|\leq \frac{1}{\mu}\|F(\hat{z}^{k+1})\|\leq \frac{1}{\mu}\|F(\tilde{z}^{k+1})\|\leq \frac{L}{\mu}\|\tilde{z}^{k+1}-z^*\|,
\end{equation*}
and so
\begin{equation*}
    \|\hat{z}^{k+1}-z^*\|\leq \frac{LL_2}{\mu^2}\|z^k-z^*\|^2,
\end{equation*}
which is exactly (\ref{eq:local}) when $z^{k+1}=\hat{z}^{k+1}$.
\end{proof}

\section{Solving the Subproblem} \label{subproblem}
In order to solve (\ref{eq:1-11}) we have to solve the system of stationarity condition in (\ref{eq:1-14}), which is equivalent to solving the following system (for $u,v$):
\begin{equation*}
\label{eq:3-1}
    \left\{\begin{array}{ll}
         \gamma\|u\|u+Q_1u+Av=b_1, \\
         \gamma\|v\|v+Q_2v-A^{\top}u=b_2.
    \end{array}\right.
\end{equation*}
where $Q_1,Q_2$ are positive definite matrices. Note that to focus on this particular system, we omit superscript $k$ and use $Q,A,b$ in place of the components of $g,H$ in (\ref{eq:1-14}) to simplify the notations.




Denote $w_1=\|u\|$ and $w_2=\|v\|$, we have:
\begin{equation*}
\label{eq:3-2}
    \left\{
    \begin{array}{ll}
         (\gamma w_1I_n+Q_1)u+Av=b_1, \\
         (\gamma w_2I_m+Q_2)v-A^{\top}u=b_2.
    \end{array}
    \right.
\end{equation*}
Therefore we get:
\begin{equation*}
\label{eq:3-3}
    \left\{
    \begin{array}{ll}
         u(w_1,w_2)=& [I_n+(\gamma w_1I_n+Q_1)^{-1}A(\gamma w_2I_m+Q_2)^{-1}A^{\top}]^{-1}\cdot \\
         & [-(\gamma w_1I_n+Q_1)^{-1}A(\gamma w_2I_m+Q_2)^{-1}b_2+(\gamma w_1I_n+Q_1)^{-1}b_1],\\
         v(w_1,w_2)=& [I_m+(\gamma w_2I_m+Q_2)^{-1}A^{\top}(\gamma w_1I_n+Q_1)^{-1}A]^{-1}\cdot\\
         & [(\gamma w_2I_m+Q_2)^{-1}A^{\top}(\gamma w_1I_n+Q_1)^{-1}b_1+(\gamma w_2I_m+Q_2)^{-1}b_2].
    \end{array}
    \right.
\end{equation*}
To simplify the notation, we let
\begin{equation}
    \label{eq:ad3-1} \left\{
    \begin{array}{lcl}
         C_u &=& I_n+(\gamma w_1I_n+Q_1)^{-1}A(\gamma w_2I_m+Q_2)^{-1}A^{\top},  \\
         d_u &=& (\gamma w_1I_n+Q_1)^{-1}\cdot(b_1-A(\gamma w_2I_m+Q_2)^{-1}b_2),  \\
         C_v &=& I_m+(\gamma w_2I_m+Q_2)^{-1}A^{\top}(\gamma w_1I_n+Q_1)^{-1}A, \\
         d_v &=& (\gamma w_2I_m+Q_2)^{-1}\cdot(b_2+A^{\top}(\gamma w_1I_n+Q_1)^{-1}b_1), \\
    \end{array}
   \right.
\end{equation}
leading to
\begin{equation}
    \label{eq:ad3-2}
    \left\{
    \begin{array}{ll}
         u(w_1,w_2) = C_u^{-1}\cdot d_u, \\
         v(w_1,w_2) = C_v^{-1}\cdot d_v. \\
    \end{array}
    \right.
\end{equation}
To further solve for $(u;v)$, one can apply Newton's method to solve the following two-variable system:
\begin{equation*}
\label{eq:3-4}
    \left\{
    \begin{array}{ll}
         u(w_1,w_2)^{\top}u(w_1,w_2) = w_1^2, \\
         v(w_1,w_2)^{\top}v(w_1,w_2) = w_2^2.
    \end{array}
    \right.
\end{equation*}
That is, to solve for the nonlinear equation system:
\begin{equation}
    \label{eq:ad3-3}
    l(w_1,w_2)=
    \left[
    \begin{array}{cc}
         \|u\|^2-w_1^2 \\
         \|v\|^2-w_2^2
    \end{array}
    \right] = 0.
\end{equation}
To apply Newton's method, the first step is to derive the Jacobian of $l(w_1,w_2)$ with the following derivation. First of all:

\begin{equation*}
    \label{eq:ad3-4}
    \frac{\partial \|u\|^2}{\partial w_1} = 2(u)^{\top}\cdot\frac{\partial u}{\partial w_1} = 2(u)^{\top}\cdot\left(\frac{\partial C_u^{-1}}{\partial w_1}d_u+C_u^{-1}\frac{\partial d_u}{\partial w_1}\right).
\end{equation*}
From (\ref{eq:ad3-1}) we have:
\begin{equation*}
    \label{eq:ad3-5}
    \frac{\partial d_u}{\partial w_1} = -\gamma (\gamma w_1I_n+Q_1)^{-2}\cdot(b_1-A(\gamma w_2I_m+Q_2)^{-1}b_2).
\end{equation*}
We also have:
\begin{equation*}
    \label{eq:ad3-6}
    \frac{\partial C_u^{-1}}{\partial w_1} = -C_u^{-1}\cdot\frac{\partial C_u}{\partial w_1}\cdot C_u^{-1},
\end{equation*}
where
\begin{equation*}
\label{eq:ad3-7}
    \frac{\partial C_u}{\partial w_1} = -\gamma(\gamma w_1I_n+Q_1)^{-2} \cdot A(\gamma w_2I_m+Q_2)^{-1}A^{\top}.
\end{equation*}
Similarly we have:
\begin{equation*}
    \label{eq:ad3-8}
    \frac{\partial \|u\|^2}{\partial w_2} = 2(u)^{\top}\cdot\frac{\partial u}{\partial w_2} = 2(u)^{\top}\cdot(\frac{\partial C_u^{-1}}{\partial w_2}d_u+C_u^{-1}\frac{\partial d_u}{\partial w_2}),
\end{equation*}
where
\begin{equation*}
    \label{eq:ad3-9}
    \frac{\partial d_u}{\partial w_2} = \gamma(\gamma w_1I_n+Q_1)^{-1}A\cdot(\gamma w_2I_m+Q_2)^{-2}\cdot b_2,
\end{equation*}
and
\begin{equation*}
    \label{eq:ad3-10}
    \frac{\partial C_u}{\partial w_2} = -\gamma(\gamma w_1I_n+Q_1)^{-1}A\cdot(\gamma w_2I_m+Q_2)^{-2}\cdot A^{\top}.
\end{equation*}

The rest of the component of the Jacobian $\frac{\partial \|v\|^2}{\partial w_1}$ and $\frac{\partial \|v\|^2}{\partial w_2}$ can be derived in a similar fashion. Being able to break down the Jacobian for $l(w_1,w_2)$, (\ref{eq:ad3-3}) can indeed be solved numerically via Newton's method efficiently.

We summarize the procedure for solving the Cubic Regularized Newton Saddle Point Subproblem (CRN-sub) with Algorithm \ref{alg:02}. Note that one could perform a spectrum transformation on $Q_1$ and $Q_2$ to transform them into diagonal matrices during the implementation if needed.
\begin{algorithm}
	\caption{Cubic Regularized Newton Saddle Point Subproblem (CRN-sub)}
	\begin{algorithmic}[1]
	\Require Initialize $w_1^0,w_2^0$; Constants $\epsilon, \gamma>0$; $Q_1,Q_2\succ 0$; $b_1$, $b_2$, $A$.
	    \State Construct $u(w_1,w_2),v(w_1,w_2)$ by (\ref{eq:ad3-1}),(\ref{eq:ad3-2}).
	    \State Construct function $l(w_1,w_2)$ with (\ref{eq:ad3-3}).
		\While{$\|l(w_1^k,w_2^k)\|>\epsilon$} {}
		    \State Compute Jacobian matrix of $l$: $J(l(w_1^k,w_2^k))$.
			\State Solve $\delta$ with  $J(l(w_1^k,w_2^k))\cdot\delta = -l(w_1^k,w_2^k)$
			\State $(w_1^{k+1};w_2^{k+1})=(w_1^k;w_2^k)+\delta$
			\State $k\leftarrow k+1$
		\EndWhile
		\State \textbf{return} $u(w_1^k,w_2^k),v(w_1^k,w_2^k)$
	\end{algorithmic}
	\label{alg:02}
\end{algorithm}


\section{Solving a Class of Convex-Concave Saddle Point Problem} \label{non-strongly}
In this section we extend our CRN method to solve a class of convex-concave sadddle point problems. Recall the saddle point problem \eqref{eq:1-1}: $$\min\limits_{x\in\mathbb{R}^n}\max\limits_{y\in\mathbb{R}^m} f(x,y).$$
For the rest of the paper, we assume $f(x,y)$ to be convex-concave only (instead of strongly convex/strongly concave). Assume however, that the set of saddle point solutions is non-empty, and that the gradient Lipschitz continuity as in Assumptions \ref{ass:2-2} and \ref{ass:2-3} hold. In this case, we focus on a class of functions $f(x,y)$ which satisfy the following error bound assumption:
\begin{assumption}[Error Bound]
\label{def:5-1}
For the function $f$ and  for any $\nu>0$, let $z^*(\nu)$ be the saddle point of the following problem
\begin{equation}
    \label{eq:7-2}
    \min\limits_{x\in\mathbb{R}^n}\max\limits_{y\in\mathbb{R}^m} f_{\nu}(x,y) := f(x,y)+\frac{\nu}{2}\|x\|^2-\frac{\nu}{2}\|y\|^2.
\end{equation}
Then there exist constants $C,\delta_0>0$, and $\theta\in(0,1]$, such that for all $0<\nu_1,\nu_2<\delta_0$, the following holds:
\begin{equation}
    \label{eq:error}
    \|z^*(\nu_1)-z^*(\nu_2)\|\leq C\cdot|\nu_1-\nu_2|^{\theta}.
\end{equation}
\end{assumption}

To gain insight into the problems satisfying Assumption \ref{def:5-1}, we shall prove below that the convex-concave quadratic models naturally satisfy Assumption~\ref{def:5-1}. 

\begin{lemma} \label{error-bound}
Let $M \in \mathbb{R}^{m\times m}$ and $b \in \mathbb{R}^m$. Suppose that $M + t I$ is invertible for all $0<t\le \delta$ where $\delta>0$ is a given constant, and
\(
L_0 := \{ x : M x = b \} \not = \emptyset.
\)
Then, there is a constant $C>0$ such that
\[
\left\| (M + t I_m)^{-1} b - (M + s I_m )^{-1} b \right\| \le C \cdot |t-s|,
\]
for all $0<t,s\le \delta$.
\end{lemma}
The proof of the above lemma can be found in Appendix~\ref{A4}.

By Lemma \ref{error-bound}, it follows that for convex-concave quadratic function $f(x,y)$, Assumption \ref{def:5-1} always holds with $\theta=1$.

\begin{proposition}
\label{prop:5-3}
Suppose that $f$ is a convex-concave quadratic function with bilinear coupling term, namely,
\begin{equation}
    \label{eq:quad}
   f(x,y)=\frac{1}{2}x^{\top}Px-b^{\top}x+x^{\top}Ay-\frac{1}{2}y^{\top}Qy+c^{\top}y,
\end{equation}
with $P,Q\succeq0$. Suppose that a stationary solution exists, then Assumption \ref{def:5-1} holds with $\theta=1$.
\end{proposition}
\begin{proof}
In view of the first-order stationarity condition of (\ref{eq:7-2}) where $f(x,y)$ is given by (\ref{eq:quad}), we have
\begin{equation*}
    \left(\begin{array}{cc}
         P+\nu I_n & A \\
         -A^{\top} & Q+\nu I_m
    \end{array}\right)\left(
    \begin{array}{cc}
         x^*(\nu)  \\
         y^*(\nu)
    \end{array}
    \right)=\left(
    \begin{array}{cc}
         b  \\
         c
    \end{array}
    \right).
\end{equation*}
Note that for $\nu>0$, $z^*(\nu)=(x^*(\nu);y^*(\nu))$ is a unique saddle point. Since a stationary solution exists, we know that
\begin{equation*}
    \left(\begin{array}{cc}
         P & A \\
         -A^{\top} & Q
    \end{array}\right)\left(
    \begin{array}{cc}
         x  \\
         y
    \end{array}
    \right)=\left(
    \begin{array}{cc}
         b  \\
         c
    \end{array}
    \right)
\end{equation*}
has a solution.
By applying Lemma \ref{error-bound} with
$M=    \left(\begin{array}{cc}
         P & A \\
         -A^{\top} & Q
    \end{array}\right)$
the error bound condition \eqref{eq:error} holds with $\theta=1$.
\end{proof}

The following lemma establishes the convergence of the sequence $\{z^*(\nu_k)\}$ under condition~\eqref{eq:error}:
\begin{lemma}
\label{lem:limit}
For a sequence $\{\nu_k\}$ such that $\lim\limits_{k\rightarrow\infty}\nu_k=0$, the sequence $\{z^*(\nu_k)\}$ has a unique limit point, denote by $z^*$:
\begin{equation*}
    \lim\limits_{k\rightarrow\infty}z^*(\nu_k)=z^*.
\end{equation*}
Furthermore, $z^*$ is a saddle point of \eqref{eq:1-1}.
\end{lemma}
\begin{proof}
Since $\{\nu_k\}$ is convergent, for any $\delta>0$, there exists a constant $K$ such that for all $i,j>K$, $|\nu_i-\nu_j|\leq\delta$. By \eqref{eq:error} we also have $\|z^*(\nu_i)-z^*(\nu_j)\|\leq C\cdot\delta^{\theta}$. Note that $C$ and $\theta$ are absolute constants, therefore $\{z^*(\nu_k)\}$ is a Cauchy sequence and has a unique limit point.

Now by the definition of $z^*(\nu_k)$, it satisfies the stationarity condition of (\ref{eq:7-2}):
\begin{equation*}
    \left\{
    \begin{array}{ll}
         \nabla_xf(z^*(\nu_k))=-\nu_kx, \\
         \nabla_yf(z^*(\nu_k))=\nu_ky.
    \end{array}
    \right.
\end{equation*}
Thus
\begin{equation*}
    \lim\limits_{k\rightarrow\infty}\nabla_xf(z^*(\nu_k))=\nabla_xf(z^*)=0 \quad\mbox{and}\quad \lim\limits_{k\rightarrow\infty}\nabla_yf(z^*(\nu_k))=\nabla_yf(z^*)=0,
\end{equation*}
which proves that $z^*$ is a saddle point of (\ref{eq:1-1}).
\end{proof}

In view of the local convergence result Theorem \ref{th:3-9} and in particular \eqref{eq:unit-quad}, if we solve the cubic regularized subproblem of (\ref{eq:7-2}) with parameter $\nu=\nu_k$ and $\gamma>0$ at $z^k=(x^k;y^k)$:
\begin{equation*}
    \label{eq:7-3}
    \begin{array}{ll}
         \min\limits_{x\in\mathbb{R}^n}\max\limits_{y\in\mathbb{R}^m} f^{\gamma}_{\nu_k}(z) & =f_{\nu_k}(z^k)+\nabla f_{\nu_k}(z^k)^{\top}(z-z^k)+\frac{1}{2} (z-z^k)^{\top}\nabla^2f_{\nu_k}(z^k)(z-z^k) \\
         & +\frac{\gamma}{3}\|x-x^k\|^3-\frac{\gamma}{3}\|y-y^k\|^3, \\
    \end{array}
\end{equation*}
and denote the solution as $z^{k+1}$, we will have the quadratic convergence for such update in the neighborhood region of $z^*(\nu_k)$. In particular, taking $\gamma=\frac{L_2\nu_k^2}{2(L+\nu_k)^2}$, and if we have $\|z^k-z^*(\nu_k)\|\leq \frac{\nu_k}{L_2}$, then we can guarantee
\begin{equation}
    \label{eq:7-4}
    \|z^{k+1}-z^*(\nu_k)\|\leq \frac{L_2}{\nu_k}\|z^k-z^*(\nu_k)\|^2=\omega_k\|z^k-z^*(\nu_k)\|^2,
\end{equation}
where $\omega_k^{-1}$ equals to the radius of quadratic convergence region $\nu_k/L_2$. Note that the specific choice of $\gamma$ here coincides with the one in Theorem \ref{th:3-9}, with the specific parameters $\nu_k,L+\nu_k,L_2$ for $f_{\nu_k}$ satisfying Assumptions \ref{ass:2-1}, \ref{ass:2-2}, \ref{ass:2-3} respectively.

Based on the error bound assumption and the quadratic convergence property of CRN method in solving (\ref{eq:7-2}), we propose a homopoty-continuation/path-following procedure to iteratively solve for an $\epsilon$-saddle point solution to (\ref{eq:1-1}). The procedure is summarized in Algorithm \ref{alg:path}:
\begin{algorithm}
	\caption{A Homotopy Continuation Procedure}
	\begin{algorithmic}[1]
	\Require Constants $L,L_2>0$. A sequence $\{\lambda_k\}$ such that $\lambda_k\in(0,1)$ for all $k$. Initialize $\nu_0>0$ and $z^0$ such that $\|z^0-z^*(\nu_0)\|\leq \frac{1}{2\omega_0}$, where $z^*(\nu_0)=(x^*(\nu_0);y^*(\nu_0))=\arg\min_x\max_y f_{\nu_0}(x,y)$.
		\While{\textbf{true}} {}
		    \State $\gamma \leftarrow \frac{L_2\nu_k^2}{2(L+\nu_k)^2}$
		    \State Solve the subproblem $z^{k+1}=(x^{k+1};y^{k+1})=\arg\min_x\max_y f^{\gamma}_{\nu_k}(x^k,y^k)$
			\State Let $\nu_{k+1}=(1-\lambda_k)\nu_k$
			\State $k\leftarrow k+1$
		\EndWhile
		\State \textbf{return} $z^k$
	\end{algorithmic}
	\label{alg:path}
\end{algorithm}


Before we proceed the discussion, we shall at this point distinguish the cases between $\theta=1$ and $\theta<1$, since these two cases result in different convergence rate as we will derive later. We first consider the case $\theta=1$.

\subsection{Error bound Assumption \ref{def:5-1} with \texorpdfstring{$\theta=1$}{theta=1}}
In order to prove the validity of Algorithm~\ref{alg:path}, there are two conditions we need to guarantee. First, we need to guarantee that after line 3, $z^{k+1}$ is still in the quadratic convergence neighborhood of the next subproblem. That is, $\|z^{k+1}-z^*(\nu_{k+1})\|\leq 1/2\omega_{k+1}$. Second, we need to ensure the termination of such procedure. That is, once $\nu_k$ is made small enough, we are able to obtain an $\epsilon$-saddle point solution to the original problem \eqref{eq:1-1}. The following two lemmas provide guarantees on these two requirements.
\begin{lemma}
\label{lem:5-3}
Suppose the sequence $\{z^k\}$ is generated by the Algorithm~\ref{alg:path}. If we set $0<\lambda_k=\lambda\leq\frac{1}{4L_2C+2}$, then
\begin{equation}
\label{eqn:new-1}
    \|z^k-z^*(\nu_k)\|\leq \frac{1}{2\omega_k},\quad \forall k\geq0.
\end{equation}
\end{lemma}
\begin{proof}
When $k=0$, \eqref{eqn:new-1} holds due to initialization requirement.  Next suppose \eqref{eqn:new-1} holds for $k = s\geq0$, then we prove that \eqref{eqn:new-1} holds for $k = s+1$. First, adding and subtracting $z^*(\nu_s)$ yields
$$\|z^{s+1}-z^*(\nu_{s+1})\|  \leq  \|z^{s+1}-z^*(\nu_s)\|+\|z^*(\nu_s)-z^*(\nu_{s+1})\|.$$
By the quadratic convergence result \eqref{eq:7-4}, $\|z^{s+1}-z^*(\nu_s)\|\leq \omega_s\|z^s-z^*(\nu_s)\|^2\leq 1/4\omega_s$. By the error bound assumption with $\theta=1$, $\|z^*(\nu_s)-z^*(\nu_{s+1})\|\leq C\lambda\nu_s$. Consequently,
$$
\|z^{s+1}-z^*(\nu_{s+1})\| \leq\frac{1}{4\omega_s}+C\lambda\nu_s \leq \nu_s\cdot\frac{4L_2C+1}{2L_2(4L_2C+2)}\leq\frac{\nu_{s}(1-\lambda)}{2L_2}=\frac{1}{2\omega_{s+1}}.
$$
Thus we prove the lemma by induction.
\end{proof}

The next lemma shows how small $\nu_k$ should be made to guarantee  an $\epsilon$-saddle point solution.
\begin{lemma}
\label{lem:5-4}
Under the setting of Lemma \ref{lem:5-3}, the sequence $\{z^*(\nu_k)\}$ converges to the limit $z^*$, which is a saddle point solution to (\ref{eq:1-1}). Furthermore, if $\nu_k\leq \frac{\epsilon}{C+\frac{1}{2L_2}}$, then $\|z^k-z^*\|\leq \epsilon$.
\end{lemma}
\begin{proof}
By Lemma \ref{lem:5-3}, we have $\lambda_k=\lambda<1$, therefore $\nu_k\rightarrow0$. Then the first statement of this lemma follows immediately from Lemma~\ref{lem:limit}. Furthermore, we could extend the error bound condition \eqref{eq:error} for $\{z^*(\nu_k)\}$ to its limit $z^*$:
\begin{equation*}
    \|z^*(\nu_k)-z^*\|\leq C\nu_k.
\end{equation*}
Therefore,
\[
         \|z^k-z^*\| \leq \|z^k-z^*(\nu_k)\|+\|z^*(\nu_k)-z^*\|   \leq \frac{1}{2\omega_k}+C\nu_k  =  \nu_k\Big(\frac{1}{2L_2}+C\Big)\leq\epsilon.
\]
\end{proof}

By Lemma~\ref{lem:5-3} and Lemma~\ref{lem:5-4}, we know that given $\nu_0>0$ and $z^0$ such that $\|z^0-z^*(\nu_0)\|\leq 1/2\omega_0$, Algorithm~\ref{alg:path} eventually generates an $\epsilon$-saddle point solution to \eqref{eq:1-1}. To start from an arbitrary initial point $z^0$, we could first initialize $\nu_0>0$ and solve \eqref{eq:7-2} with CRN-SPP (Algorithm~\ref{alg:01}) to a desirable neighborhood region of $z^*(\nu_0)$ (phase 1), then followed by implementing Algorithm~\ref{alg:path} (phase 2). Let us call this method a {\em homotopy-continuation Cubic Regularized Newton method for Saddle Point Problem}\/ (hc-CRN-SPP), and summarize the procedure in Algorithm~\ref{alg:03}.


\begin{algorithm}
	\caption{hc-CRN-SPP}
	\begin{algorithmic}[1]
	\Require Constants $L,L_2,\bar{\gamma},C,D,\epsilon>0$; $\rho\in(0,1)$; $\tau\in(0,1]$,  $\lambda=\frac{\tau}{2+4L_2C}$. Initialize $z^0=(x^0;y^0)$ and $\nu_0>0$.
	    \State \textbf{Start Phase 1}
        \State $\epsilon_1\leftarrow\frac{\nu_0^4}{8L_2^2}$, $L_m=(L+\nu_0)^2+(L+\nu_0)L_2D$, and $\alpha\leftarrow\frac{\nu_0^2}{2L_m}$.
        \State $z^k =$ CRN-SPP($f_{\nu_0},z^0,\epsilon_1,\bar{\gamma},\rho,\alpha$)
	    \State \textbf{End Phase 1}
	    \State \textbf{Start Phase 2}
		\Require: Initialize $\bar{z}^0\leftarrow z^k$. $\|\bar{z}^0-z^*(\nu_0)\|\leq \frac{1}{2\omega_0}=\frac{\nu_0}{2L_2}$, where $z^*(\nu_0)=(x^*(\nu_0);y^*(\nu_0))=\arg\min_x\max_y f_{\nu_0}(x,y)$.
		\While{$\nu_k>\frac{\epsilon}{C+\frac{1}{2L_2}}$} {}
		    \State $\gamma\leftarrow\frac{L_2\nu_k^2}{2(L+\nu_k)^2}$
		    \State Solve the subproblem $\bar{z}^{k+1}=(\bar{x}^{k+1};\bar{y}^{k+1})=\arg\min_x\max_y f^{\gamma}_{\nu_k}(\bar{x}^k,\bar{y}^k)$
			\State Let $\nu_{k+1}=(1-\lambda)\nu_k$
			\State $k\leftarrow k+1$
		\EndWhile
		\State \textbf{End Phase 2}
		\State \textbf{return} $\bar{z}^k$
	\end{algorithmic}
	\label{alg:03}
\end{algorithm}


The next Theorem gives the iteration complexity bound in order to establish an $\epsilon$-saddle point solution to problem \eqref{eq:1-1} with Algorithm~\ref{alg:03}.
\begin{theorem}
\label{th:5-5}
For the class of convex-concave saddle point problems \eqref{eq:1-1} satisfying Assumption~\ref{def:5-1}, Algorithm \ref{alg:03} returns an $\epsilon$-saddle point solution $\bar{z}^k$. It takes $K_1$ iterations for CRN-SPP to return $z^k$ in phase 1 and takes $K_2$ iterations to terminate phase 2. In particular, we have
\begin{equation*}
    K_1 =\mathcal{O}\left(\frac{(L+\nu_0)(L+\nu_0+L_2)}{\nu_0^2}\ln\left(\frac{8L_2^2}{\nu_0^4}\right)\right),
\end{equation*}
and
\begin{equation*}
    K_2 = \mathcal{O}\left((1+2L_2C)\cdot\ln\left(\frac{\nu_0(C+\frac{1}{2L_2})}{\epsilon}\right)\right).
\end{equation*}
\end{theorem}

\begin{proof}
We first show that after $K_1$ iterations, the precision requirement $\epsilon_1=\frac{\nu_0^4}{8L_2^2}$ is met for CRN-SPP and the output $z^k$ satisfies $\|z^k-z^*(\nu_0)\|\leq\frac{\nu_0}{2L_2}$ as required at the start of phase 2.

As shown in Theorem \ref{th:3-7}, given a precision $\epsilon_1$, a function $f_{\nu_0}$ with convexity modulus $\nu_0$ and Lipschitz constant $L+\nu_0$ and $L_2$ (for $F(z)$ and $\nabla F(z)$ respectively), the required iterations is:
\begin{equation*}
    \mathcal{O}\left(\left(\kappa^2+\kappa\cdot\frac{L_2}{\mu}\right)\ln\left(\frac{1}{\epsilon_1}\right)\right)=\mathcal{O}\left(\frac{(L+\nu_0)(L+\nu_0+L_2)}{\nu_0^2}\ln\left(\frac{8L_2^2}{\nu_0^4}\right)\right).
\end{equation*}
In addition, this indicates the output $z^k$ has $m(z^k)=\frac{1}{2}\|F(z^k)\|^2\leq\epsilon_1$, then we have:
\begin{equation*}
    \|z^k-z^*(\nu_0)\|\leq \frac{1}{\nu_0}\|F(z^k)\|\leq \frac{\sqrt{2\epsilon_1}}{\nu_0}=\frac{\nu_0}{2L_2}.
\end{equation*}
Finally, letting
\begin{equation*}
    K_2=\frac{2+4L_2C}{\tau}\cdot\ln\left(\frac{\nu_0(C+\frac{1}{2L_2})}{\epsilon}\right)
    =\frac{1}{\lambda}\cdot\ln\left(\frac{\nu_0(C+\frac{1}{2L_2})}{\epsilon}\right),
\end{equation*}
we have:
\begin{equation*}
    \nu_{K_2}=(1-\lambda)^{K_2}\nu_0\leq \exp(-K_2\lambda)\nu_0=\frac{\epsilon}{C+\frac{1}{2L_2}},
\end{equation*}
which by Lemma \ref{lem:5-4}, guarantees the output $\bar{z}^k$ such that $\|\bar{z}^k-z^*\|\leq\epsilon$.
\end{proof}

\begin{remark}
The overall complexity for  Algorithm~\ref{alg:03} is linear, with phase 1 depends on parameter $[(L+\nu_0)(L+\nu_0+L_2)]/\nu_0^2$ and phase 2 depends on $1+2L_2C$. Note that the proposed scheme should not be confused with algorithms solving general convex-concave saddle point problems and admitting sub-linear convergence rate. The focus here is on the class of problem satisfying the error bound condition in Assumption~\ref{def:5-1} with $\theta=1$.
\end{remark}

We shall now continue the discussion for the case $\theta<1$.

\subsection{Error bound Assumption \ref{def:5-1} with \texorpdfstring{$\theta<1$}{theta<1}}
Same as the case for $\theta=1$, we need to derive conditions similar to Lemma \ref{lem:5-3}, Lemma \ref{lem:5-4} to guarantee that Algorithm \ref{alg:path} works. The first main difference between these two cases presents in the parameter choice of $\lambda_k$. For $\theta<1$, we can no longer use a constant $\lambda$ throughout the iterations. In fact, the choice of $\lambda_k$ needs to depend on $\nu_k$ as shown below.

\begin{lemma}
\label{lem:5-9}
Suppose the sequence $\{z^k\}$ is generated by the Algorithm \ref{alg:path}. If we set $\lambda_k$ such that,
\begin{equation}
    \label{eq:lambda_k}
    0<\lambda_k^{\theta}\leq \frac{\nu_k^{1-\theta}}{4L_2C+2\nu_k^{1-\theta}},
\end{equation}
then
\begin{equation}
\label{eqn:new-k}
    \|z^k-z^*(\nu_k)\|\leq \frac{1}{2\omega_k},\quad \forall k\geq0.
\end{equation}
\end{lemma}
\begin{proof}
When $k=0$, \eqref{eqn:new-k} holds due to initialization requirement.  Next suppose \eqref{eqn:new-k} holds for $k = s\geq0$. We shall then prove \eqref{eqn:new-k} holds for $k = s+1$. First, adding and subtracting $z^*(\nu_s)$ yields
$$\|z^{s+1}-z^*(\nu_{s+1})\|  \leq  \|z^{s+1}-z^*(\nu_s)\|+\|z^*(\nu_s)-z^*(\nu_{s+1})\|.$$
By the quadratic convergence result \eqref{eq:7-4}, $\|z^{s+1}-z^*(\nu_s)\|\leq \omega_s\|z^s-z^*(\nu_s)\|^2\leq 1/4\omega_s$. By the error bound assumption, $\|z^*(\nu_s)-z^*(\nu_{s+1})\|\leq C(\lambda_s\nu_s)^{\theta}$. Consequently,
\begin{equation*}
    \begin{array}{ll}
        \|z^{s+1}-z^*(\nu_{s+1})\| & \leq \frac{1}{4\omega_s}+C(\lambda_s\nu_s)^{\theta} \leq \nu_s\cdot\frac{4L_2C+\nu_s^{1-\theta}}{2L_2(4L_2C+2\nu_s^{1-\theta})} \\
         & \leq\frac{\nu_{s}(1-\lambda_s^{\theta})}{2L_2}\leq\frac{\nu_{s}(1-\lambda_s)}{2L_2}=\frac{1}{2\omega_{s+1}}.
    \end{array}
\end{equation*}
The lemma is proven by induction.
\end{proof}

The following lemma shows that $\{\nu_k\}$ converges to $0$ for the choice of $\lambda_k$ in (\ref{eq:lambda_k}) and provides an upper bound for the convergence rate. The proof is relegated to Appendix \ref{A5}.

\begin{lemma}
\label{lem:5-10}
Let the sequence $\{\nu_k\}$ be generated by Algorithm \ref{alg:path} with $\lambda_k$ satisfying (\ref{eq:lambda_k}). Then for a constant $K$, we have $\nu_K<1$. Moreover, the rest of the sequence converges to $0$ at a sublinear rate:
\begin{equation*}
    \nu_k\leq \left(\frac{1}{1+C'\cdot k}\right)^{\frac{\theta}{1-\theta}},\quad\mbox{for all $k\geq K$},
\end{equation*}
where
  $  C'=\frac{1-\theta}{\theta}\cdot\frac{1}{(4L_2C+2)^{\frac{1}{\theta}}}$.
\end{lemma}

By Lemma \ref{lem:5-10} we know that $\{\nu_k\}$ converges to $0$ even for $\theta<1$, which is a crucial fact since it guarantees the limit of the sequence $z^*(\nu_k)$  (namely $z^*$) is a saddle point of (\ref{eq:1-1}), by Lemma \ref{lem:limit}.

The next lemma stipulates the value of $\nu_k$ to generate an $\epsilon$-saddle point solution.
\begin{lemma}
\label{lem:5-11}
Under the setting of Lemma \ref{lem:5-9}, the sequence $\{z^*(\nu_k)\}$ converges to the limit $z^*$, which is a saddle point solution to (\ref{eq:1-1}). Furthermore, if $\nu_k^{\theta}\leq \frac{\epsilon}{C+\frac{1}{2L_2}}$, then $\|z^k-z^*\|\leq \epsilon$.
\end{lemma}
\begin{proof}
By Lemma \ref{lem:5-10}, we have $\nu_k\rightarrow0$. Then the first statement of this lemma follows immediately from Lemma~\ref{lem:limit}. Furthermore, we could extend the error bound condition \eqref{eq:error} for $\{z^*(\nu_k)\}$ to its limit $z^*$:
\begin{equation*}
    \|z^*(\nu_k)-z^*\|\leq C\nu_k^{\theta}.
\end{equation*}
Therefore,
\begin{equation*}
\begin{array}{ll}
    \|z^k-z^*\| & \leq \|z^k-z^*(\nu_k)\|+\|z^*(\nu_k)-z^*\|   \leq \frac{1}{2\omega_k}+C\nu_k^{\theta} \\
     & =\nu_k^{\theta}\Big(\frac{\nu_k^{1-\theta}}{2L_2}+C\Big)\leq\nu_k^{\theta}\Big(\frac{1}{2L_2}+C\Big) \leq\epsilon,
\end{array}
\end{equation*}
where we assume $\nu_k<1$ without loss of generality.
\end{proof}

An algorithm similar to Algorithm \ref{alg:03} is proposed for $\theta<1$. The main difference is that $\lambda_k$ is no longer a constant and needs to be updated at each iteration. We shall call it hc/$\theta$-CRN-SPP.
\begin{algorithm}
	\caption{hc/$\theta$-CRN-SPP}
	\begin{algorithmic}[1]
	\Require Constants $L,L_2,\bar{\gamma},C,D,\epsilon>0$; $\rho,\theta\in(0,1)$.
     Initialize $z^0=(x^0;y^0)$ and $\nu_0>0$.
	    \State \textbf{Start Phase 1}
        \State $\epsilon_1\leftarrow\frac{\nu_0^4}{8L_2^2}$, $L_m=(L+\nu_0)^2+(L+\nu_0)L_2D$, and $\alpha\leftarrow\frac{\nu_0^2}{2L_m}$.
        \State $z^k =$ CRN-SPP($f_{\nu_0},z^0,\epsilon_1,\bar{\gamma},\rho,\alpha$)
	    \State \textbf{End Phase 1}
	    \State \textbf{Start Phase 2}
		\Require Initialize $\bar{z}^0\leftarrow z^k$. $\|\bar{z}^0-z^*(\nu_0)\|\leq \frac{1}{2\omega_0}=\frac{\nu_0}{2L_2}$, where $z^*(\nu_0)=(x^*(\nu_0);y^*(\nu_0))=\arg\min_x\max_y f_{\nu_0}(x,y)$. Constant $\tau\in(0,1]$.
		\While{$\nu_k^{\theta}>\frac{\epsilon}{C+\frac{1}{2L_2}}$} {}
		    \State $\gamma\leftarrow\frac{L_2\nu_k^2}{2(L+\nu_k)^2}$
		    \State Solve the subproblem $\bar{z}^{k+1}=(\bar{x}^{k+1};\bar{y}^{k+1})=\arg\min_x\max_y f^{\gamma}_{\nu_k}(\bar{x}^k,\bar{y}^k)$
		    \State Set $\lambda_k=\tau\cdot\left(\frac{\nu_k^{1-\theta}}{4L_2C+2\nu_k^{1-\theta}}\right)^{\frac{1}{\theta}}$, $\nu_{k+1}=(1-\lambda_k)\nu_k$.
			\State $k\leftarrow k+1$
		\EndWhile
		\State \textbf{End Phase 2}
		\State \textbf{return} $\bar{z}^k$
	\end{algorithmic}
	\label{alg:04}
\end{algorithm}

The final theorem establishes the iteration complexity result for Algorithm~\ref{alg:04} to generate an $\epsilon$-saddle point solution to (\ref{eq:1-1}). Note that the required iteration number $K_1$ in phase 1 is the same as for the case $\theta=1$ in Theorem~\ref{th:5-5}.

\begin{theorem}
\label{th:5-12}
For the class of convex-concave saddle point problems \eqref{eq:1-1} satisfying Assumption~\ref{def:5-1} with $\theta<1$, Algorithm \ref{alg:04} returns an $\epsilon$-saddle point solution $\bar{z}^k$. It takes $K_1$ iterations for CRN-SPP to return $z^k$ in phase 1 and takes $K_2$ iterations to terminate phase 2, where
\begin{equation*}
    K_1 =\mathcal{O}\left(\frac{(L+\nu_0)(L+\nu_0+L_2)}{\nu_0^2}\ln\left(\frac{8L_2^2}{\nu_0^4}\right)\right),
\end{equation*}
and
\begin{equation*}
    K_2 = \mathcal{O}\left(\left(\frac{\theta}{1-\theta}\cdot(4L_2C+2)^{\frac{1}{\theta}}\right)\cdot\left(\frac{C+\frac{1}{2L_2}}{\epsilon}\right)^{\frac{1-\theta}{\theta^2}}\right).
\end{equation*}
\end{theorem}
\begin{proof}
The proof for $K_1$ is exactly the same as in Algorithm~\ref{alg:03}.
Let $K_2$ be
\begin{equation*}
 \left(\frac{\theta}{1-\theta}\cdot(4L_2C+2)^{\frac{1}{\theta}}\right)\cdot\left(\left(\frac{C+\frac{1}{2L_2}}{\epsilon}\right)^{\frac{1-\theta}{\theta^2}}-1\right)=\frac{1}{C'}\cdot\left(\left(\frac{C+\frac{1}{2L_2}}{\epsilon}\right)^{\frac{1-\theta}{\theta^2}}-1\right).
\end{equation*}
Without loss of generality, let as assume $\nu_0<1$, then by Lemma~\ref{lem:5-10}, we have:
\begin{equation*}
    \nu_{K_2}\leq \left(\frac{1}{1+C'\cdot K_2}\right)^{\frac{\theta}{1-\theta}}=\left(\frac{\epsilon}{C+\frac{1}{2L_2}}\right)^{\frac{1}{\theta}}.
\end{equation*}
By Lemma~\ref{lem:5-11}, this guarantees the output $\bar{z}^k$ to be so that $\|\bar{z}^k-z^*\|\leq \epsilon$.
\end{proof}

\section{Numerical Experiments} \label{numerical}
We consider the following saddle point problem in our experiments:
\begin{equation}
    \label{eq:6-1}
    \begin{array}{rcl}
         \min\limits_{x\in\mathbb{R}^n}\max\limits_{y\in\mathbb{R}^m}f(x,y) & = & \frac{1}{M_1}\sum\limits_{i=1}^{M_1}\ln(1+e^{-a_i^{\top}x})+\frac{1}{2}\|x\|^2 \\
         & & +x^\top Ay-\frac{1}{M_2}\sum\limits_{j=1}^{M_2}\ln(1+e^{-b_j^{\top}y})-\frac{1}{2}\|y\|^2,
    \end{array}
\end{equation}
where both strongly-convex part and strongly-concave part consist of logistic functions with regularization, together with a bilinear coupling. This simple model will help confirm/validate our convergence results for CRN-SPP and provide comparison with other first-order methods such as Extra-Gradient (EG) and Optimistic Gradient Descent Ascent (OGDA).

We provide a few notes on the implementation details. This experiment is conducted under Matlab 2018a environment. For the problem size, we set $n=100$, $m=200$, $M_1=M_2=1000$, and $a_i$, $b_j$, $A$ are generated randomly with Matlab built-in function randn($\cdot,\cdot$) for corresponding dimensions. The step size $\alpha$ is manually tuned in a certain range for best performance for each method. In this experiment we set $\alpha=0.1$ for CRN-SPP, $\alpha=0.04$ for EG, and $\alpha=0.02$ for OGDA. We set $\bar{\gamma}=1$. However, to avoid additional computation need in each iterations we simply take $\gamma^k=\min(\bar{\gamma},\frac{3\mu^2}{4b})$ instead of repeatedly decreasing $\gamma^k$ as suggested in Algorithm \ref{alg:01}, which requires solving additional subproblems. Note that the number $\frac{3\mu^2}{4b}$ comes from the upper bound derived in \eqref{eq:3-8}, where $\mu=1$ in this experiment and $b=\max(\|\nabla_xf(x^k,y^k)\|,\|\nabla_yf(x^k,y^k)\|)$.

In Figure~\ref{fig:compare-merit} we show the convergence of our proposed CRN-SPP, together with EG and OGDA methods in terms of the merit function $m(z)=\frac{1}{2}\|F(z)\|^2$ in log scale. One can see that CRN-SPP converges relatively fast to high precision within about 15 iterations with quadratic convergence during the process. For EG and OGDA methods, little convergence is observed in the first 30 iterations. It takes around 950/1900 iterations for EG/OGDA to reach similar precision as CRN-SPP. However, it does take significantly longer at each iteration for CRN-SPP, which is on average around 50 times longer than EG/OGDA. Considering total run time, we can conclude that the proposed CRN-SPP is comparable with the other two methods. Figure \ref{fig:compare:distance} shows the convergence in terms of the distance to saddle point solution, which presents similar convergence behavior to merit function. The saddle point is reached by running CRN-SPP for 30 iterations.

Figure \ref{fig:sub} shows the convergence of solving CRN subproblems with Algorithm \ref{alg:02}. Note that in each subproblem we are solving the nonlinear equation system \eqref{eq:ad3-3}, thus the convergence is in terms of $\|l(w_1,w_2)\|$. In each subproblem we initialize $(w_1^0,w_2^0)=(0.5,0.5)$. The figure shows that under different problem sizes, the iteration numbers required for $\|l(w_1,w_2)\|$ to converge to precision $10^{-5}$. The x-axis shows the dimension $n$ of variable $x$, while the dimension $m$ of $y$ is given by $m=1.5*n$. $M_1=M_2=1000$ are fixed parameters. 100 experiments are conducted for each different problem size. Finally, Figure \ref{fig:gamma} shows how choosing different $\bar{\gamma}$ could affect the convergence speed. In the previous 3 figures we use $\bar{\gamma}=1$, while in this figure $\bar{\gamma}$ varies among $[0,0.5,1,1.5,2]$. Figure \ref{fig:gamma} shows that in this experiment, smaller $\bar{\gamma}$ results in faster convergence speed. In particular, $\bar{\gamma}=0$ indicates a pure Newton update instead of cubic regularized Newton update. It should be noted that we did not explicitly derive the convergence analysis for pure Newton update, and the question that whether there are scenarios where CRN update outperforms pure Newton update is left to future work.

\begin{figure}[htbp]
\centering
\begin{minipage}[t]{0.48\textwidth}
\centering
\includegraphics[width=6cm]{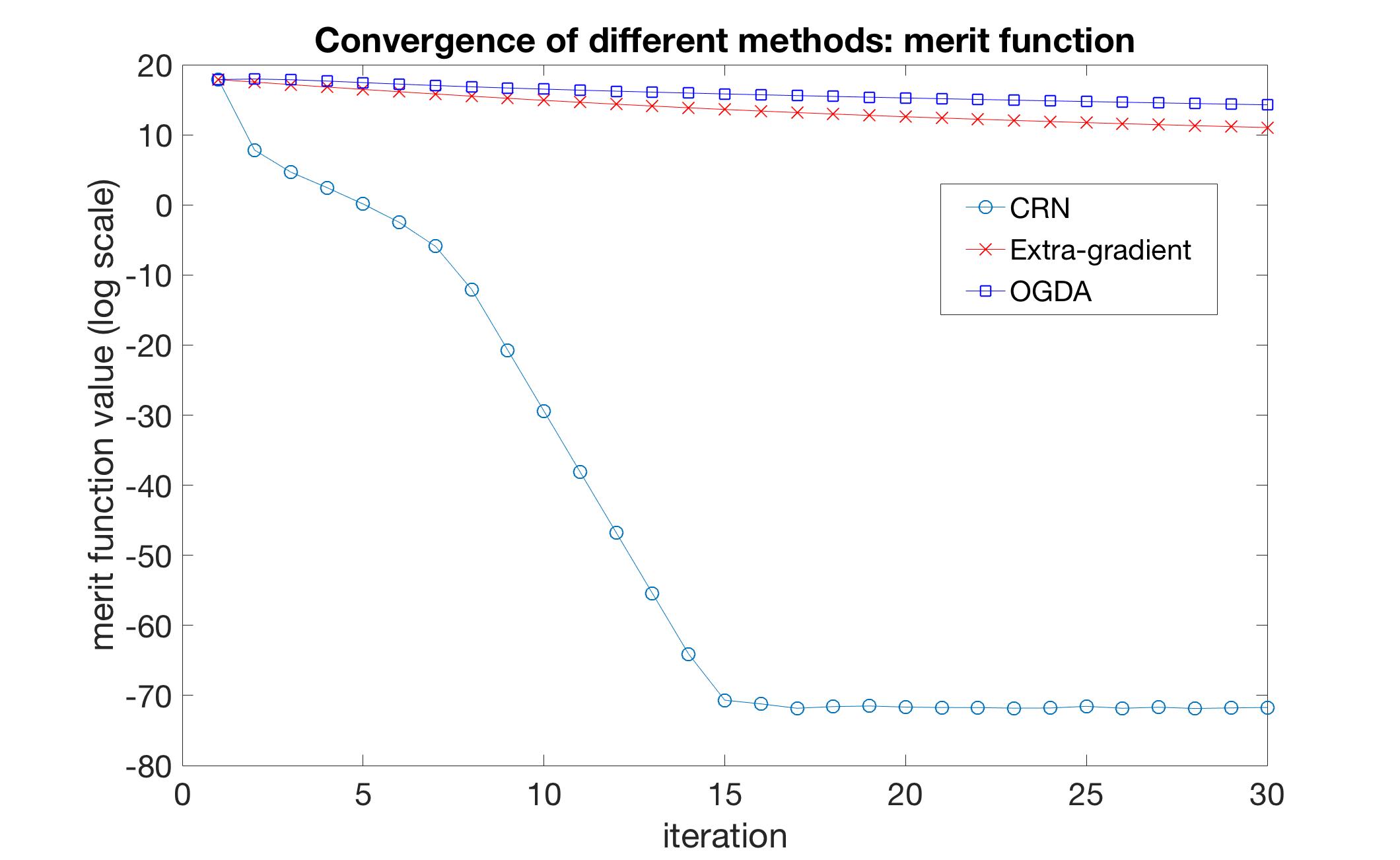}
\caption{Convergence of merit function}
\label{fig:compare-merit}
\end{minipage}
\begin{minipage}[t]{0.48\textwidth}
\centering
\includegraphics[width=6cm]{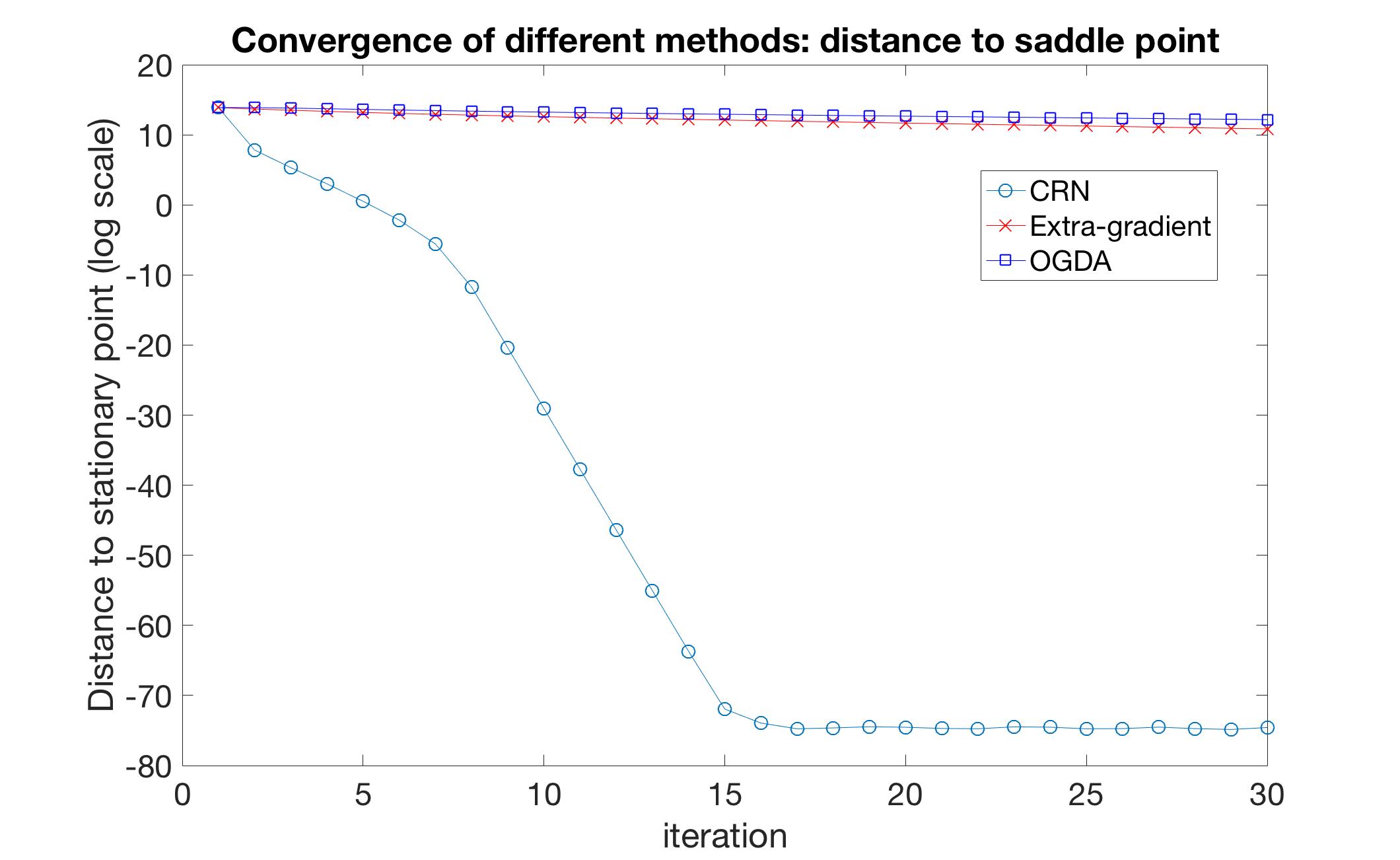}
\caption{Convergence of distance to saddle point}
\label{fig:compare:distance}
\end{minipage}
\end{figure}

\begin{figure}[htbp]
\centering
\begin{minipage}[t]{0.48\textwidth}
\centering
\includegraphics[width=6cm]{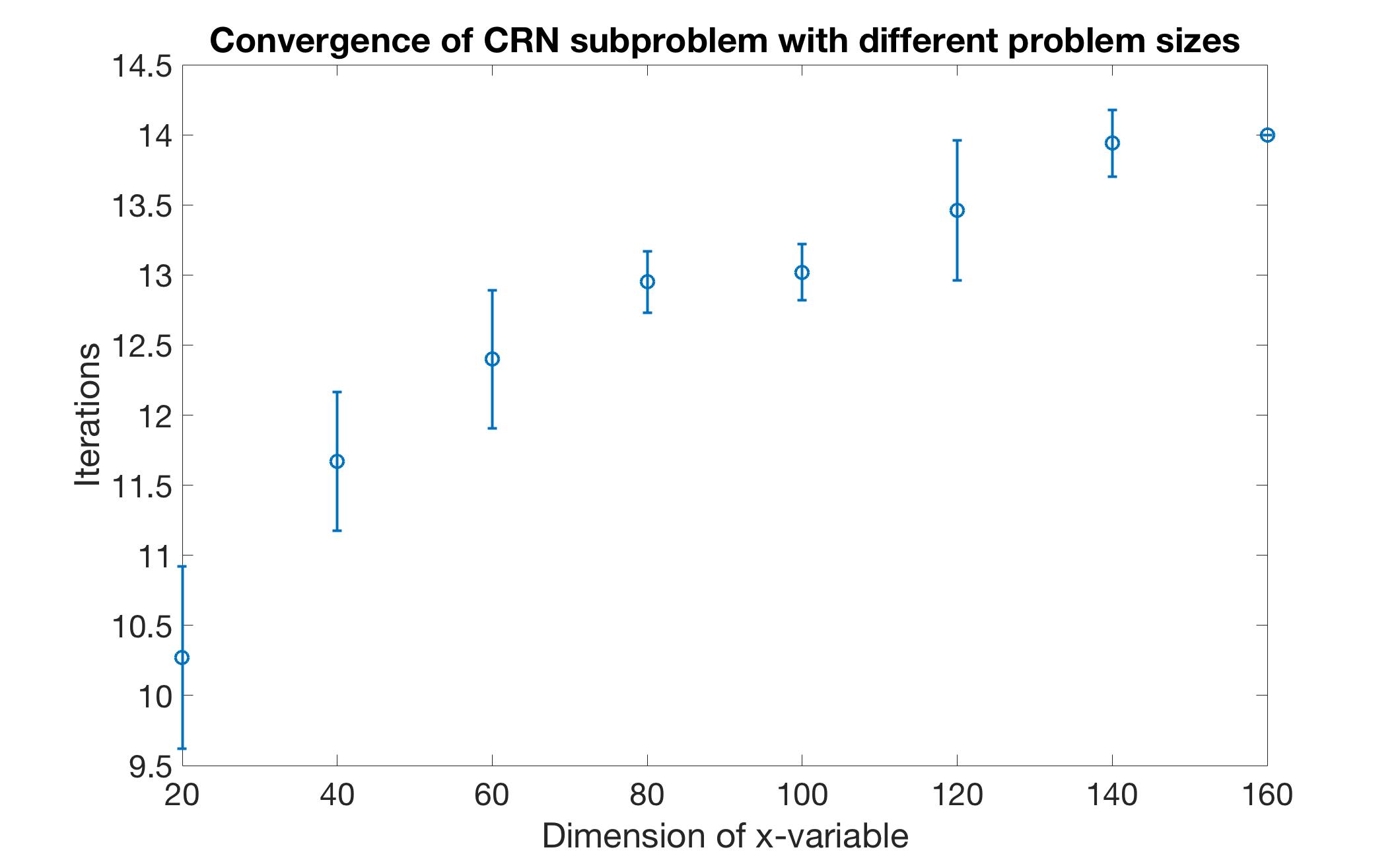}
\caption{Convergence of CRN subproblem}
\label{fig:sub}
\end{minipage}
\begin{minipage}[t]{0.48\textwidth}
\centering
\includegraphics[width=6cm]{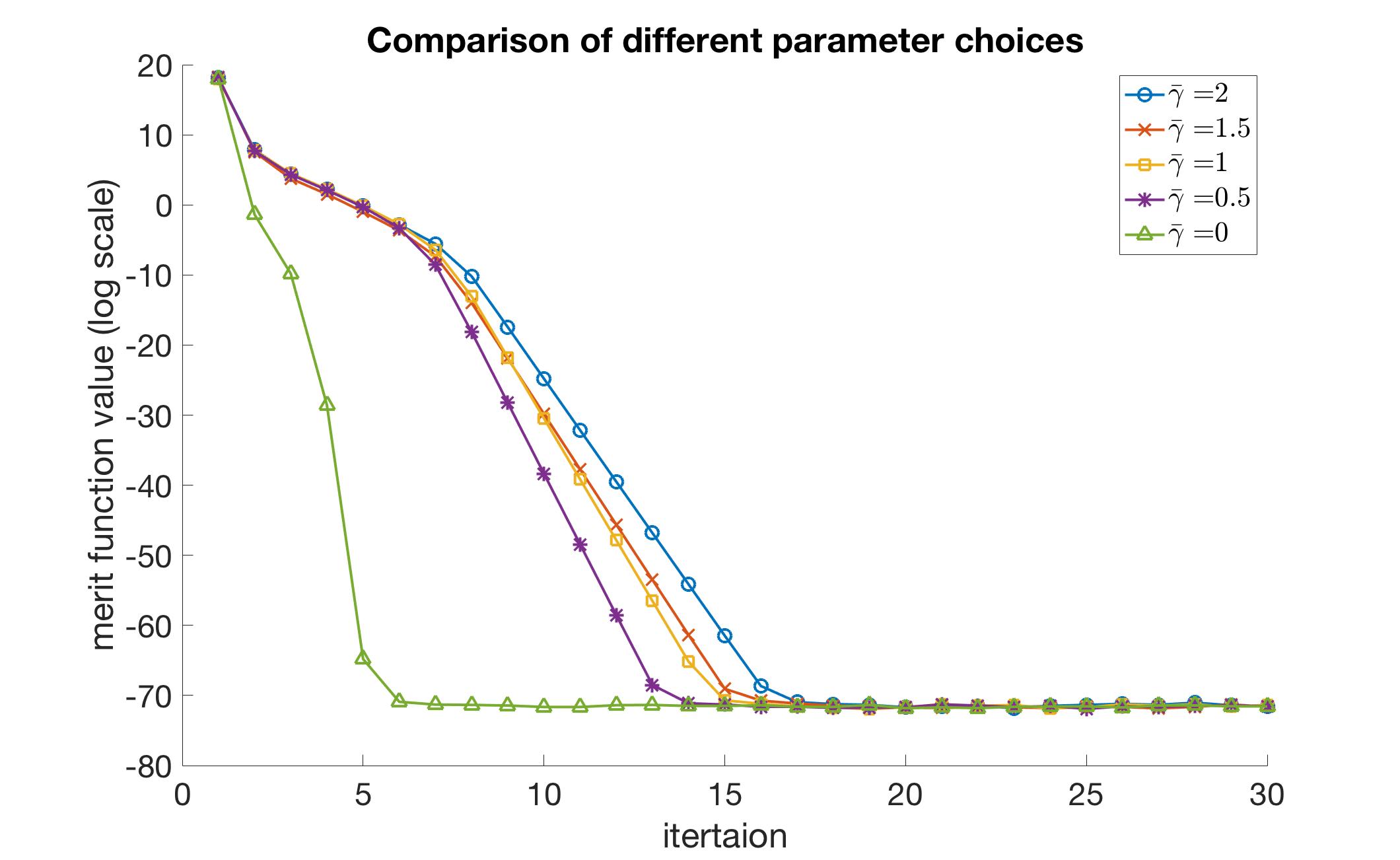}
\caption{Comparison of different $\bar{\gamma}$}
\label{fig:gamma}
\end{minipage}
\end{figure}



\section{Conclusions} \label{conclusion}

\subsection{ Conservativeness of the merit function} \label{MF}

In Section~\ref{CRN-sub} we discuss the dependency of this method on the parameters is $\kappa^2$, which is worse than the dependence of $\kappa$ for the first order methods in the literature. 
At first appearance, it presents a major setback for the second order method. However, the worsened complexity bound is due to the choice of the merit function. In this context, the merit function $m(z)$ is a double-sided sword. On the one hand, it allows one to establish a complexity bound for CRN-SPP. On the other hand, $m(z)$ as a measure of progress is overly conservative. To illustrate this point, below we shall present an analysis
for the extra-gradient method and optimistic gradient descend ascend method using this merit function; the analysis of these two methods leading to the optimal complexity bound is taken from~\cite{mokhtari2019unified}.

A vanilla extra-gradient (EG) method is given by:
\begin{equation}
   \label{eq:add2-21}
    \begin{array}{ll}
         & z^{k+1/2} = z^k - \eta F(z^k), \\
         & z^{k+1} = z^k - \eta F(z^{k+1/2}), \\
    \end{array}
\end{equation}
where $\eta$ is the step size and $F(z)=\left(\nabla_xf(x,y);-\nabla_yf(x,y)\right)$.
To analyze this method with the merit function $m(z)$, we need to check whether this update is descent regarding $m(z^k)$, as in the first step in our analysis: 
\begin{equation*}
    \langle\nabla m(z^k),z^{k+1}-z^k\rangle=-\eta\langle\nabla F(z^k)^{\top}F(z^k),F(z^{k+1/2})\rangle\leq -c<0,
\end{equation*}
for some positive constant $c$.

Based on the update rule in (\ref{eq:add2-21}) we have
   \[
          F(z^{k+1/2}) 
          = F(z^k)-\eta\nabla F(z^k)F(z^k)
          -\eta\int_{t=0}^1\left(\nabla F(z^k-t\eta F(z^k))-\nabla F(z^k)\right)F(z^k)dt.
  \]
Therefore,
\begin{eqnarray}
    \label{eq:add2-24}
          &   & -\langle\nabla F(z^k)^{\top}F(z^k),F(z^{k+1/2})\rangle \nonumber \\
          & = & -F(z^k)^{\top}\nabla F(z^k)F(z^k)+\eta F(z^k)^{\top}\nabla F(z^k)\nabla F(z^k)F(z^k) \nonumber \\
          &   & + \eta F(z^k)^{\top}\nabla F(z^k)\left(\int_{t=0}^1((\nabla F(z^k-t\eta F(z^k))-\nabla F(z^k))dt)F(z^k) \right) \nonumber  \\
          & \leq & -\mu\|F(z^k)\|^2+\eta L^2\|F(z^k)\|^2+\frac{\eta^2LL_2}{2}\|F(z^k)\|^3.
\end{eqnarray}
For \eqref{eq:add2-24} to be less than 0, we need to take $\eta<\frac{\mu}{L^2}$. Compared to the choice of $\eta=\frac{1}{4L}$ in \cite{mokhtari2019unified}, if we look at the final step of the proof of Theorem 7 in \cite{mokhtari2019unified} (equation (130)), we have:
\begin{equation*}
    \|z^{k+1}-z^*\|^2\leq (1-\eta\mu)\|z^k-z^*\|^2-(1-\eta^2L^2-2\eta\mu)\|z^k-z^{k+1/2}\|^2.
\end{equation*}
A choice of $\eta=\frac{1}{4L}$ results in:
\begin{equation*}
    \|z^{k+1}-z^*\|^2\leq (1-\frac{1}{4\kappa})\|z^k-z^*\|^2.
\end{equation*}

On the other hand, if we were to guarantee the descent direction in terms of the merit function $m(z)$, we would have to choose $\eta$ in the order of $\frac{\mu}{L^2}$. For example, a choice of $\eta=\frac{\mu}{4L^2}$ will result in:
\begin{equation*}
    \|z^{k+1}-z^*\|^2\leq (1-\frac{1}{4\kappa^2})\|z^k-z^*\|^2,
\end{equation*}
which actually gives the same convergence complexity $\mathcal{O}\left(\kappa^2\ln(1/\epsilon)\right)$ as in our scheme.

To analyze the Optimistic Gradient Descent Ascent (OGDA) method and Proximal Point (PP) method, which also has an iteration complexity bound of  $\mathcal{O}\left(\kappa\ln(1/\epsilon)\right)$ following the analysis in~\cite{mokhtari2019unified}.
In fact, it was established in~\cite{mokhtari2019unified} that both EG and OGDA are approximations of the PP method, in the sense that:
\begin{equation}
    \label{eq:add2-28}
    \|z^{k+1}-\hat{z}^{k+1}\|\leq o(\eta^2),
\end{equation}
where $z^{k+1}$ is the next iterate with OGDA/EG method and $\hat{z}^{k+1}$ is the next iterate with PP method, given the same current iterate $z^k$ and a positive stepsize $\eta$ (see Propositions 1, 2 in \cite{mokhtari2019unified}).

Now with a little abuse of notation, let us redefine $z^{k+1}$ as the next iterate of EG, and let $\hat{z}^{k+1}$ be the next iterate by either OGDA/PP method. Therefore, for the update direction for OGDA/PP to be gradient-related to $m(z^k)$, we have:
\begin{equation}
    \label{eq:add2-29}
    \begin{array}{ll}
        \langle\nabla m(z^k),\hat{z}^{k+1}-z^k\rangle & = \langle\nabla m(z^k),(z^{k+1}-z^k)+(\hat{z}^{k+1}-z^{k+1})\rangle \\
         & \leq -\mu\|F(z^k)\|^2+\eta L^2\|F(z^k)\|^2 + o(\eta^2), \\
    \end{array}
\end{equation}
where we use the results from \eqref{eq:add2-24} and \eqref{eq:add2-28}. That is, the sign of \eqref{eq:add2-29} will be dominated by the first two terms with small $\eta$. This results in the choice of $\eta<\frac{\mu}{L^2}$ to satisfy the gradient-related requirement, leading again to an iteration complexity bound of $\mathcal{O}\left(\kappa^2\ln(1/\epsilon)\right)$.

\subsection{Concluding remarks}
In this paper we develop a cubic regularized Newton (CRN) method to solve unconstrained convex-concave saddle point problems. We first consider a general strongly-convex-strongly-concave saddle point problem, where at each iteration we build a CRN model as a local approximation of the original function and solve the corresponding saddle point subproblem for an update direction. We then adopt a constant step size to guarantee the theoretical decrease in the merit function. We propose to use the squared norm of the gradient as a merit function to measure the progress of such CRN update. A global convergence with iteration complexity $\mathcal{O}\left((\kappa^2+\kappa\cdot\frac{L_2}{\mu})\ln(1/\epsilon)\right)$ and local quadratic convergence are established. We also provide analysis on solving the CRN saddle point subproblems. We propose to incorporate a homotopy continuation/path-following procedure for solving a class of convex-concave saddle point problems that satisfies a certain error bound assumption. Finally, numerical experiments are conducted which confirm the convergence behavior of the proposed CRN-SPP method.
Possible future research includes improvements of global convergence rate in terms of the dependency on condition number $\kappa$, while retaining local superlinear convergence rate. As we discussed in the previous subsection,
this will likely need to rely on different (or unified) merit functions.
Another future research topic is to develop adaptive strategies to implement CRN-SPP, so as to require no knowledge on the problem parameters such as $\mu$, $L$ and $L_2$  {\it a priori}.




\begin{appendices}
\section{Proofs of the Propositions and Theorems}
\subsection{Proof of Proposition \ref{th:2-6}} \label{A1}
With Assumption \ref{ass:2-1}, we have
$$
f(x,y^*)-f(x^*,y^*)\geq\frac{\mu}{2}\|x-x^*\|^2\qquad\mbox{and}\qquad f(x^*,y^*)-f(x^*,y)\geq\frac{\mu}{2}\|y-y^*\|^2.
$$
As a result,
    $\frac{\mu}{2}\Big(\|x-x^*\|^2 + \|y-y^*\|^2\Big)\leq f(x,y^*)-f(x^*,y)$.
Denote $z = (x;y)$ and $z^* = (x^*;y^*)$. By the Lipschitzian Assumption \ref{ass:2-2}, it holds that
   $ \|F(z)\|^2=\|F(z)-F(z^*)\|^2\leq L^2\|z-z^*\|^2$,
which leads to the first half of our result 
\begin{equation*}
         m(z) \leq \frac{L^2}{2}\|z-z^*\|^2 \leq \frac{L^2}{\mu}(f(x,y^*)-f(x^*,y)) \leq \frac{L^2}{\mu}\left(\max\limits_{y'\in\mathbb{R}^m} f(x,y')-\min\limits_{x'\in\mathbb{R}^n} f(x',y)\right).
\end{equation*}
On the other hand, denote
 \[
          y^*(x)=\arg\max\limits_{y'\in\mathbb{R}^m}f(x,y') \qquad\mbox{and}\qquad
          x^*(y)=\arg\min\limits_{x'\in\mathbb{R}^n}f(x',y).
 \]
With this notation, the duality gap can be rewritten as  $f(x,y^*
(x))-f(x^*(y),y)$. By the first-order stationarity condition, we have
 \[   \nabla_xf(x^*(y),y)=0\qquad\mbox{and}\qquad \nabla_yf(x,y^*(x))=0. \]
Applying the Lipschitz continuity condition yields
\begin{equation*}
    \label{eq:add2-15}
    \begin{array}{ll}
         f(x,y) & \leq f(x^*(y),y)+\nabla_xf(x^*(y),y)^{\top}(x^*(y)-x)+\frac{L}{2}\|x^*(y)-x\|^2 \\
         & =f(x^*(y),y)+\frac{L}{2}\|x^*(y)-x\|^2.\\
    \end{array}
\end{equation*}
Similarly,
  $  f(x,y)\geq f(x,y^*(x))-\frac{L}{2}\|y^*(x)-y\|^2$.
Combining these two yields
\begin{equation}
    \label{eq:add2-17}
    \begin{array}{ll}
         f(x,y^*(x))-f(x^*(y),y) & = f(x,y^*(x))-f(x,y)+f(x,y)-f(x^*(y),y) \\
         & \leq \frac{L}{2}\left(\|x^*(y)-x\|^2+\|y^*(x)-y\|^2\right). \\
    \end{array}
\end{equation}
Additionally, the strong convexity/strong concavity of $f$ gives
\begin{equation*}
    \label{eq:add2-18}
    \begin{array}{ll}
         & \|\nabla_xf(x,y)\|^2=\|\nabla_xf(x,y)-\nabla_xf(x^*(y),y)\|^2\geq \mu^2\|x-x^*(y)\|^2, \\
         & \|\nabla_yf(x,y)\|^2=\|\nabla_yf(x,y)-\nabla_yf(x,y^*(x))\|^2\geq \mu^2\|y-y^*(x)\|^2, \\
    \end{array}
\end{equation*}
resulting
\begin{equation}
    \label{eq:add2-19}
    \begin{array}{ll}
         \mu^2(\|x^*(y)-x\|^2+\|y^*(x)-y\|^2)
         \leq \|\nabla_xf(x,y)\|^2+\|\nabla_yf(x,y)\|^2  \leq 2m(z). \\
    \end{array}
\end{equation}
Combining (\ref{eq:add2-17}),(\ref{eq:add2-19}) we the second half of the result:
 \[   \max\limits_{y'\in\mathbb{R}^m} f(x,y')-\min\limits_{x'\in\mathbb{R}^n} f(x',y)=f(x,y^*
(x))-f(x^*(y),y)\leq\frac{L}{\mu^2}m(z). \]

\subsection{Proof of Proposition \ref{pro:3-3}} \label{A2}
We first prove that \eqref{eq:cond} can be achieved with small enough $\gamma^k$. Note that $u^k,v^k$ are the solutions to the stationarity condition \eqref{eq:1-14}, which is equivalent to the following system:
\begin{equation}
    \label{eq:a2-1}
   \begin{cases}
    \gamma^k\|u^k\|u^k+Q_1u^k+H_{xy}^kv^k=-g_x^k,\\
    \gamma^k\|v^k\|v^k+Q_2v^k-(H_{xy}^k)^{\top}\!u^k\!=\!g_y^k,
   \end{cases}
\end{equation}
where $Q_1 = H_{xx}^k \succeq\mu I$ and $Q_2 = -H_{yy}^k \succeq \mu I$ are positive definite matrices. Inner product the first equation in \eqref{eq:a2-1} with $u^k$ and the second with $v^k$ and then sum up the two, we get
 \[   \gamma^k\big(\|u^k\|^3+\|v^k\|^3\big)+(u^k)^{\top}Q_1u^k+(v^k)^{\top}Q_2v^k= - (g_x^k)^{\top}u^k+(g_y^k)^{\top}v^k. \]
Consequently, let $b = \max\big\{\|g_x^k\|,\|g_y^k\|\big\}$, we have
\[ \gamma^k(\|u^k\|^3+\|v^k\|^3)+\mu(\|u^k\|^2+\|v^k\|^2)\leq {b}(\|u^k\|+\|v^k\|).\]
Note that $\|u^k\|^2+\|v^k\|^2\geq\frac{1}{2}(\|u^k\|+\|v^k\|)^2$ and $\|u^k\|^3+\|v^k\|^3\geq\frac{1}{4}(\|u^k\|+\|v^k\|)^3$. As a result,
 \[   \gamma^k(\|u^k\|+\|v^k\|)^3+2\mu(\|u^k\|+\|v^k\|)^2\leq4{b}(\|u^k\|+\|v^k\|).\]
Let $\omega = \gamma^k(\|u^k\|+\|v^k\|)$, then the above inequality is equivalent to $\omega^2 + 2\mu\omega - 4b\gamma^k\leq0$. Solving this quadratic inequality yields that
\begin{equation}
\label{eq:3-8}
    \gamma^k(\|u^k\|+\|v^k\|)\leq \sqrt{\mu^2+4{b}\gamma^k}-\mu = \frac{4b\gamma^k}{\sqrt{\mu^2+4{b}\gamma^k}+\mu}\to 0 \quad\mbox{as}\quad \gamma^k\to0.
\end{equation}
We can see that the upper bound for $\gamma^k(\|u^k\|+\|v^k\|)$ is an increasing function of $\gamma^k$ with function values ranging from 0 to $\infty$. This indicates that by making $\gamma^k$ small enough, condition \eqref{eq:cond} can then be satisfied.

\smallskip
Next, we proceed to prove descent result of Proposition \ref{pro:3-3}. The proof of this part is based on the concept of the proof in \cite{taji1993globally}. By direct calculation,
\begin{eqnarray}
\label{eq:2-5}
         \langle \nabla m(z^k),d^k\rangle& = &\langle H_{xx}^kg_x^k+H_{xy}^kg_y^k,u^k\rangle  + \langle H_{yy}^kg_y^k+(H_{xy}^k)^{\top}g_x^k,v^k\rangle.
\end{eqnarray}
The first term on the RHS of \eqref{eq:2-5} can be written as:
\begin{eqnarray}
\label{eq:descent-1}
& & \langle H_{xx}^kg_x^k+H_{xy}^kg_y^k,u^k\rangle +\mu\cdot(u^k)^{\top}H_{xy}^kv^k+(g_x^k)^{\top}H_{xy}^kv^k-(u^k)^{\top}H_{xy}^kg_y^k\\
& = & \langle g_x^k+H_{xx}^ku^k+H_{xy}^kv^k, g_x^k+\mu\cdot u^k\rangle -\|g_x^k\|^2-\mu\cdot (u^k)^{\top}H_{xx}^ku^k -\mu\cdot(g_x^k)^{\top}u^k\nonumber\\
& \overset{\text{By} (\ref{eq:1-14})}{=} & \langle -\gamma^k\|u^k\|u^k,g_x^k+\mu\cdot u^k\rangle -\|g_x^k\|^2-\mu\cdot(u^k)^{\top}H_{xx}^ku^k-\mu\cdot(g_x^k)^{\top}u^k\nonumber\\
& \leq & -\mu^2\|u^k\|^2-\mu\gamma^k\|u^k\|^3-\gamma^k\|u^k\|(u^k)^{\top}g_x^k-\|g_x^k\|^2-\mu\cdot(g_x^k)^{\top}u^k \nonumber\\
&\overset{(i)}{\leq}& -\left(\frac{\mu^2}{2}+\mu\gamma^k\|u^k\|-\frac{1}{2}(\gamma^k)^2\|u^k\|^2\right)\|u^k\|^2\nonumber\\
& \overset{(ii)}{\leq} & -\frac{\mu^2}{2}\|u^k\|^2\nonumber
\end{eqnarray}
where inequality (ii) is because $\gamma^k\|u^k\|\leq\mu$ and inequality (i) is due to
\begin{eqnarray*}  & & -\gamma^k\|u^k\|(u^k)^{\top}g_x^k-\|g_x^k\|^2-\mu\cdot(g_x^k)^{\top}u^k\\
& = &  -\left(\frac{1}{2}\|g_x^k\|^2+\gamma^k\|u^k\|(u^k)^{\top}g_x^k\right)-\left(\frac{1}{2}\|g_x^k\|^2+\mu(g_x^k)^{\top}u^k\right)\\
 & \leq & \frac{1}{2}(\gamma^k)^2\|u^k\|^4 + \frac{\mu^2}{2}\|u^k\|^2.
\end{eqnarray*}
Similarly, for the second term of \eqref{eq:2-5}, we have
\begin{eqnarray*}
 &  &\langle H_{yy}^kg_y^k+(H_{xy}^k)^{\top}g_x^k,v^k\rangle -\mu\cdot(u^k)^{\top}H_{xy}^kv^k-(g_x^k)^{\top}H_{xy}^kv^k+(u^k)^{\top}H_{xy}^kg_y^k  \\
 & = &\langle -g_y^k-H_{yy}^kv^k-(H_{xy}^k)^{\top}u^k, -g_y^k+\mu\cdot v^k\rangle-\|g_y^k\|^2+\mu\cdot(v^k)^{\top}H_{yy}^kv^k+\mu\cdot(g_y^k)^{\top}v^k \\
 & \leq & -\frac{\mu^2}{2}\|v^k\|^2.
\end{eqnarray*}
Adding \eqref{eq:descent-1} to the above inequality, 
combining with \eqref{eq:2-5}, we have 
   $ \langle\nabla m(z^k),d^k\rangle\leq-\frac{\mu^2}{2}\|d^k\|^2$,
which completes the proof.

\subsection{Proof of Theorem \ref{th:3-7}} \label{A3}
First of all, we establish the descent lemma for the mapping $F(z)$ by observing that:
\begin{equation*}
\label{eq:2-17}
    F(z^{k+1})=F(z^k+\alpha d^k)=F(z^k)+\alpha\nabla F(z^k)d^k+\int^1_{t=0}(\nabla F(z^k+t\alpha d^k)-\nabla F(z^k))\alpha d^kdt.
\end{equation*}
Then with Assumption \ref{ass:2-2}, we have the following inequality:
\begin{equation}
\label{eq:2-18}
    \begin{array}{ll}
         \|F(z^{k+1})\| & \leq \|F(z^k)+\alpha\nabla F(z^k)d^k\|+\alpha\|d^k\|\int_{t=0}^1\|\nabla F(z^k+t\alpha d^k)-\nabla F(z^k)\|dt \\
         & \leq \|F(z^k)+\alpha\nabla F(z^k)d^k\|+\frac{\alpha^2L_2}{2}\|d^k\|^2.
    \end{array}
\end{equation}
We can rewrite the expression for $\nabla F(z^k)d^k$ using stationarity condition \eqref{eq:1-14}:
\begin{equation*}
\label{eq:2-19}
    \nabla F(z^k)d^k=
    \begin{pmatrix}
    H_{xx}^ku^k + H_{xy}^kv^k \\
    -(H_{xy}^k)^{\top}u^k-H_{yy}^kv^k
    \end{pmatrix}
    = - F(z^k) -
    \gamma^k\begin{pmatrix}
    \|u^k\|u^k \\
    \|v^k\|v^k
    \end{pmatrix}.
\end{equation*}
Putting the above identity back to \eqref{eq:2-18} yields
\begin{equation}
\label{eq:2-20}
\begin{array}{lcl}
     \|F(z^{k+1})\|& \leq & (1-\alpha)\|F(z^k)\|+\alpha\gamma^k\big(\|u^k\|^2+\|v^k\|^2\big)+\frac{\alpha^2L_2}{2}\|d^k\|^2\\
    & \leq & (1-\alpha)\|F(z^k)\| + \big(\alpha\bar\gamma + \frac{\alpha^2L_2}{2}\big)\|d^k\|^2.\nonumber
\end{array}
\end{equation}
Note that \eqref{eq:2-13} indicates that
$\|d^k\|^2 \leq \frac{8L_m}{\mu^4}(m(z^k)-m(z^{k+1}))=\frac{4}{\alpha\mu^2}(m(z^k)-m(z^{k+1}))$, which further yields
\begin{equation}
    \label{eq:2-22-1}
    \begin{array}{ll}
         \|d^k\|^2 & \leq \frac{4}{\alpha\mu^2}(m(z^k)-m(z^{k+1}))  \\
         & = \frac{2}{\alpha\mu^2}(\|F(z^k)\|^2-\|F(z^{k+1})\|^2) \\
         & = \frac{2}{\alpha\mu^2}(\|F(z^k)\|+\|F(z^{k+1})\|)(\|F(z^k)-F(z^{k+1})\|) \\
         & \leq \frac{4L D}{\alpha\mu^2}(\|F(z^k)\|-\|F(z^{k+1})\|).
    \end{array}
\end{equation}
where the last inequality is due to $
\|F(z)\| = \|F(z)-F(z^*)\|\leq L D$ for $\forall z\in\{z:m(z)\leq m(z^0)\}.$ Define $\beta = \left(\frac{L_2}{L_m}+4\Bar{\gamma}/\mu^2\right)L D$. Then combining \eqref{eq:2-22-1} and \eqref{eq:2-20} yields that
\begin{equation*}
    \label{eq:2-22-2}
    \begin{array}{ll}
         \|F(z^{k+1})\| & \leq  (1-\alpha)\|F(z^k)\|+\beta(\|F(z^k)\|-\|F(z^{k+1})\|), \\
    \end{array}
\end{equation*}
which results in:
\begin{equation}
\label{eq:2-24}
    \|F(z^{k+1})\|\leq \frac{(1-\alpha)+\beta}{1+\beta}\|F(z^k)\|=(1-\frac{\alpha}{1+\beta})\|F(z^k)\|.
\end{equation}
Squaring both sides of \eqref{eq:2-24} and dividing by half, we get the desired bound
\begin{equation*}
    m(z^{k+1})\leq\Big(1-\frac{\alpha}{1+\beta}\Big)^2m(z^k).
\end{equation*}
Finally, taking $\bar{\gamma}= \frac{L_2\mu^2}{4L_m}$, we have $\beta=(\frac{L_2}{L_m}+4\Bar{\gamma}/\mu^2)\kappa D=\frac{2L_2L D}{L_m}=\frac{2L_2LD}{L^2+L_2LD}\leq2$, which further yields
\begin{equation*}
m(z^{k+1})\leq\Big(1-\frac{\alpha}{3}\Big)^2m(z^k)=\left(1-\frac{\mu^2}{6L_m}\right)^2m(z^k).
\end{equation*}

\subsection{Proof of Lemma \ref{error-bound}} \label{A4}

Suppose $r=\mbox{\rm rank }(M)$. If $r=m$ (namely $M$ is invertible) then the lemma holds true trivially. Now, suppose $r<m$, and let a singular value decomposition of $M$ be
\[
M=U^\top \Lambda V, \mbox{ where } \Lambda =\left( \begin{array}{cc} \Lambda_r, & 0_{r\times (m-r)} \\ 0_{(m-r)\times r}, & 0_{(m-r)\times (m-r)} \end{array}\right),
\]
with $\Lambda_r$ being an $r\times r$ diagonal positive, and $U$ and $V$ are orthonormal matrices.
The pseudo-inverse of $M$ is $M^+=V^\top \Lambda^+ U$, where $\Lambda^+=\left( \begin{array}{cc} \Lambda_r^{-1}, & 0_{r\times (m-r)} \\ 0_{(m-r)\times r}, & 0_{(m-r)\times (m-r)} \end{array}\right)$. According to the theory of pseudo-inverse matrices (cf.~\cite{BG03}), $L_0=\{x: Mx=b\}\not=\emptyset$ if and only if $b=MM^+b$, or equivalently, the last $m-r$ elements of $U b$ are zero; that is, $U b=\left( \begin{array}{c} \bar{b}_r  \\ 0_{m-r} \end{array}\right)$.

Now, let $G=UV^\top$, which is also orthonormal, and introduce
\begin{eqnarray*}
x_t &:=& (M+t I)^{-1} b = (U^\top \Lambda V + t I)^{-1} b = V^\top (\Lambda + t G )^{-1} U b \\
& =&  V^\top \left( \begin{array}{cc} \Lambda_r + t G_{11} & t G_{12}  \\ t G_{21} & t G_{22} \end{array}\right)^{-1} U b.
\end{eqnarray*}
In fact, observe that $G_{22}$ is invertible. To see this, notice that $\det(M+t I_m)$ is exactly of the order $t^{m-r}$. However, if $G_{22}$ would be degenerate, then
\[
\det \left( \begin{array}{cc} \Lambda_r + t G_{11} & t G_{12}  \\ t G_{21} & t G_{22} \end{array}\right) = \det (\Lambda_r + t G_{11}) \cdot
\det \left( G_{22} - t G_{21} ( \Lambda_r + t G_{11})^{-1} G_{12} \right) \cdot t^{m-r}
\]
is at least of the order $O(t^{m-r+1})$ for sufficiently small $t>0$, which is a contradiction. Therefore, $G_{22}$ must be invertible.

In general, consider a $2\times 2$ invertible block matrix $\left( \begin{array}{cc} A & B  \\ C & D \end{array}\right)$,
where $A$ and its Schur complement $D-CA^{-1}B$ are invertible. Then (see~\cite{LS02}),
\begin{eqnarray*}
& &
\left( \begin{array}{cc} A & B  \\ C & D \end{array}\right)^{-1} \\
&=&
\left( \begin{array}{cc}
A^{-1}+A^{-1}B(D-CA^{-1}B)^{-1}CA^{-1} & -A^{-1}B(D-CA^{-1}B)^{-1} \\
-(D-CA^{-1}B)^{-1}CA^{-1} & (D-CA^{-1}B)^{-1} \end{array}\right).
\end{eqnarray*}
Substituting $A=\Lambda_r + t G_{11}$, $B=t G_{12}$, $C=t G_{21}$ and $D=t G_{22}$ into the above expression, we have
\[
\left( \begin{array}{cc} \Lambda_r + t G_{11} & t G_{12}  \\ t G_{21} & t G_{22} \end{array}\right)^{-1}
=\left( \begin{array}{cc}
\Lambda_r^{-1} + O(t) & -\Lambda_r^{-1} G_{12} G_{22}^{-1} + O(t) \\
-G_{22}^{-1} G_{21} \Lambda_r^{-1} + O(t) & G_{22}^{-1}/t + O(t) \end{array}\right).
\]
Therefore,
\[
x_t = V^\top \left( \begin{array}{c} \Lambda_r^{-1} \bar{b} + O(t) \\ -G_{22}^{-1} G_{21} \Lambda_r^{-1} \bar{b} + O(t) \end{array}\right).
\]
Since $L_0=\{x: Mx=b\}\not=\emptyset$, it follows that $M^+b \in L_0$. Let $x_0:=M^+b = V^\top \left( \begin{array}{c} \Lambda_r^{-1} \bar{b} \\ -\Lambda_r^{-1} G_{12} G_{22}^{-1} \bar{b} \end{array}\right) \in L_0$, we have $\| x_t - x_0\|=O(t)$. By the smoothness of the curve $\{x_t: 0<t<\delta\}$, we actually have $\|x_t-x_s\|=O(|t-s|)$ for sufficiently small positive $t$ and $s$.

\subsection{Proof of Lemma~\ref{lem:5-10}}
\label{A5}
First of all, note that since $0<\lambda_k<1$, the sequence $\{\nu_k\}$ is strictly decreasing. Let us first assume $\nu_0>1$. Then for $k<K$ such that $\nu_k>1$, we can take $\lambda_k=\lambda$ as a constant:
\begin{equation*}
    \lambda = \left(\frac{1}{4L_2C+2}\right)^{\frac{1}{\theta}}\leq \left(\frac{\nu_k^{1-\theta}}{4L_2C+2\nu_k^{1-\theta}}\right)^{\frac{1}{\theta}}.
\end{equation*}
Indeed, for
   $ K>(4L_2C+2)^{\frac{1}{\theta}}\cdot\ln\nu_0$,
we have
   $ \nu_K = (1-\lambda)^K\nu_0\leq \exp(-K\lambda)\nu_0<1$.

Let us now focus on the case when $\nu_k<1$. Without loss of generality, take $\lambda_k$ as its upper bound in (\ref{eq:lambda_k}).
Therefore we have 
\begin{equation*}
    \nu_{k+1}=\left(1-\left(\frac{\nu_k^{1-\theta}}{4L_2C+2\nu_k^{1-\theta}}\right)^{\frac{1}{\theta}}\right)\nu_k=\nu_k-\frac{\nu_k^{\frac{1}{\theta}}}{(4L_2C+2\nu_k^{1-\theta})^{\frac{1}{\theta}}}\leq \nu_k-\xi\cdot\nu_k^{\frac{1}{\theta}},
\end{equation*}
where $\xi<1$ is a constant defined as
\begin{equation*}
    \frac{1}{\xi}:=(4L_2C+2)^{\frac{1}{\theta}}\geq (4L_2C+2\nu_k^{1-\theta})^{\frac{1}{\theta}},
\end{equation*}
for $\nu_k<1$.

Therefore, to establish the convergence of $\{\nu_k\}$, we could instead establish the convergence of the following sequence:
\begin{equation}
    \label{eq:A5-1}
    a_0<1,\quad a_{k+1}=a_k-\xi\cdot a_k^{\frac{1}{\theta}},
\end{equation}
for $\theta\in(0,1)$.

The remaining part of this proof follows from the proof of Theorem 1 in \cite{nesterov2019inexact}.

Let us first note that the function
   $ f(x)=\frac{1}{(1+x)^p}$
is convex for $x\geq-1$ and $p>0$. Therefore, for $x\geq-1$, we have
    $f(x)=\frac{1}{(1+x)^p}\geq f(0)+f'(0)x=1-px$.
Taking $p=\frac{1-\theta}{\theta}>0$ and $x=\frac{a_{k+1}-a_k}{a_k}>-1$, we obtain
\begin{equation}
    \label{eq:A5-2}
    \left(\frac{a_k}{a_{k+1}}\right)^{\frac{1-\theta}{\theta}}=\frac{1}{\left(1+\frac{a_{k+1}-a_k}{a_k}\right)^{\frac{1-\theta}{\theta}}}\geq 1-\frac{1-\theta}{\theta}\cdot\frac{a_{k+1}-a_k}{a_k}.
\end{equation}
Then
\begin{equation*}
    a_{k+1}^{\frac{\theta-1}{\theta}}-a_{k}^{\frac{\theta-1}{\theta}}=a_{k}^{\frac{\theta-1}{\theta}}\left(\frac{a_{k+1}^{\frac{\theta-1}{\theta}}}{a_{k}^{\frac{\theta-1}{\theta}}}-1\right)\overset{(\ref{eq:A5-2})}{\geq}\frac{1-\theta}{\theta}\cdot\frac{a_k-a_{k+1}}{a_k^{\frac{1}{\theta}}}\overset{(\ref{eq:A5-1})}{=} \frac{1-\theta}{\theta}\cdot\xi.
\end{equation*}
Summing up the above inequality from $a_0$ to $a_k$, we have
\begin{equation*}
    a_k^{\frac{\theta-1}{\theta}}\geq a_0^{\frac{\theta-1}{\theta}}+\frac{1-\theta}{\theta}\cdot k\xi\geq 1+\frac{1-\theta}{\theta}\cdot k\xi.
\end{equation*}
Therefore,
    $a_k\leq \left(\frac{1}{1+\frac{1-\theta}{\theta}\cdot k\xi}\right)^{\frac{\theta}{1-\theta}}$.
By the definition of $C'$ in Lemma~\ref{lem:5-10}, we obtain
    $\nu_k\leq \left(\frac{1}{1+C'\cdot k}\right)^{\frac{\theta}{1-\theta}}$,
for all $k$ such that $\nu_k<1$.

\end{appendices}
\end{document}